\documentclass{amsart}

\usepackage{graphicx}
\usepackage{enumerate}
\usepackage{amsmath,amssymb,amscd}
\usepackage{mathrsfs}
\usepackage{color}
\usepackage[margin=20pt,font=small,labelfont=bf,textfont=it]{caption}
\usepackage{ascmac}
\usepackage{here}
\usepackage{marginnote}
\usepackage{latexsym}
\usepackage {autobreak}

\newtheorem*{TLP}{Takens' Last Problem}

\newtheorem{mthm}{Theorem}

\newtheorem{subthm}{Theorem}

\newtheorem{thm}{Theorem}[section]
\newtheorem{lem}[thm]{Lemma}
\newtheorem{prop}[thm]{Proposition}

\newtheorem{clam}[thm]{Claim}

\theoremstyle{definition}
\newtheorem{rmk}[thm]{Remark}

\numberwithin{equation}{section}
\numberwithin{figure}{section}

\def\smfd{W^{\mathrm{s}}}
\def\umfd{W^{\mathrm{u}}}
\def\lsmfd{W^{\mathrm{s}}_{\mathrm{loc}}}

\def\lumfd{W^{\mathrm{u}}_{\mathrm{loc}}}

\def\Int{\mathrm{Int}}
\def\Im{\mathrm{Image}}

\def\setl{\setlength{\leftskip}{-18pt}}

\begin{document}

\title[
Historic and physical wandering domains
]
{
Historic and physical wandering domains for 
wild blender-horseshoes
}

\author{Shin Kiriki}
\address[Shin Kiriki]{Department of Mathematics, Tokai University, 4-1-1 Kitakaname, Hiratuka, Kanagawa, 259-1292, JAPAN}
\email{kiriki@tokai-u.jp}

\author{Yushi Nakano}
\address[Yushi Nakano]{Department of Mathematics, Tokai University, 4-1-1 Kitakaname, Hiratuka, Kanagawa, 259-1292, JAPAN}
\email{yushi.nakano@tsc.u-tokai.ac.jp}

\author{Teruhiko Soma}
\address[Teruhiko Soma]{Department of Mathematical Sciences, Tokyo Metropolitan University, 1-1 Minami-Ohsawa, Hachioji, Tokyo, 192-0397, JAPAN}
\email{tsoma@tmu.ac.jp}


\date{\today}


\begin{abstract}
We present diffeomorphisms of wild blender-horseshoes 
which belong to $C^r$ $(1\leq r<\infty)$ closures of 
two types of diffeomorphisms, 
one of which has a historic contracting wandering domain, 
and 
the other has a non-trivial Dirac physical measure  supported by saddle periodic orbit.
It is a non-trivial extension of Colli-Vargas' model \cite{CV01}  
to the higher dimensional dynamics with the use of wild blender-horseshoes.
\end{abstract}
\maketitle

\section{Introduction}\label{s.Introduction}
\subsection{Historicity and physicality}
Ruelle and Takens focused on dynamics which is irregular under Lebesgue measure by introducing the concept of ``historic behaviour" in \cite{R01,T08}. 
Here,  for a map $f$ on a Riemannian manifold $M$, we say that $f$ or some orbit of $f$ has \emph{historic behaviour} 
if there is $x\in M$ such that the forward orbit $\left\{f^{i}(x):i\geq 0\right\}$ has non-converging Birkhoff averages.
That is,  
\[
\frac{1}{n+1}\sum_{i=0}^n\delta_{f^i(x)} 
\]
dose not converge as $n\to\infty$
in the weak*-topology, 
where $\delta_{f^i(x)}$ is the Dirac measure on $M$ supported at $f^i(x)$. 
There have been sporadic studies of examples of the non-existence of Birkhoff average, while the following open questions have been 
proposed in order to study dynamical historicity from a unified point of view.

\begin{TLP}[\cite{T08}]
Are there persistent classes of smooth 
dynamical systems such that the set of points whose orbits have historic
behaviour has positive Lebesgue measure?
\end{TLP}
 
Affirmative answers to the problem  are already provided for non-hyperbolic situations.
To explain it, 
recall the non-hyperbolic phenomenon given by Newhouse \cite{N79} that,  
for any $C^{2}$ diffeomorphism on a smooth manifold $M$ of $\dim M=2$
with a homoclinic tangency of a saddle periodic point, 
there is an open set $\mathcal{N} \subset \mathrm{Diff}^{2}(M)$ whose closure contains $f$ 
and such that any $g\in \mathcal{N}$  
has a $C^{2}$-robust homoclinic tangency of some  
hyperbolic sets $\Lambda_g$ which is homoclinically related to the 
continuation of the saddle periodic point. 
We say that such a $C^2$-open set of nonhyperbolic diffeomorphisms is  a ($C^{2}$-)\emph{Newhouse open set} or \emph{domain}. 
An affirmative answer was detected in any $C^{2}$-Newhouse open set by the use of non-trivial wandering domains as follows.
Here the \emph{non-trivial wandering domain} for $f$ means a non-empty connected open set $\mathbb{D}\subset M$ with the following conditions: 
\begin{itemize}
\item $f^{i}(\mathbb{D})\cap f^{j}(\mathbb{D})=\emptyset$ for any integers $i, j\geq 0$ with $i\neq j$; 
\item the union of $\omega$-limit sets of all $x\in\mathbb{D}$,  $\omega(\mathbb{D}, f)=\bigcup_{x\in \mathbb{D}}\omega(x,f)$, is not equal to a single periodic orbit.
\end{itemize}
A wandering domain $\mathbb{D}$ is called \emph{contracting} 
if the diameter of $f^{i}(\mathbb{D})$ 
converges to zero as $i\to +\infty$. 
Existence of contracting non-trivial  wandering domains 
was first brought to light by Colli-Vargas \cite{CV01} for a prototype of 
wild hyperbolic set, that is, an affine thick horseshoe with homoclinic tangencies. 
Kiriki and Soma \cite{KS17}
not only generalised their results in any $C^{2}$-Newhouse open set, 
but also identified the presence of historic behaviour as follows:
 any Newhouse open set in $\mathrm{Diff}^{r}(M)$, $\dim M=2$ and $2\leq r<\infty$, contains a dense subset 
every element of which 
 has a \emph{historic wandering domain}, 
i.e.~a non-trivial wandering domain of points whose orbits have historic behaviour. 
Note that arguments in \cite{KS17} are not extendable to $C^\infty$-diffeomorphisms, 
but Berger and Biebler \cite{BB} overcome this difficulty with completely different methods.

The arguments of \cite{N79} would not be applicable to $C^1$-diffeomorphisms on 2-dimensional 
manifold $M$. In fact, 
$\mathrm{Diff}^1(M)$ with $\dim M=2$ contains a generic subset 
where every diffeomorphism has no homoclinic tangency \cite{M11}. Thus any method similar 
to that in \cite{KS17} might be irrelevant for $\mathrm{Diff}^1(M)$ 
if $\dim M=2$. On the other hand, if $\dim M\geq 3$,
Bonatti and D\'{\i}az \cite{BD12} presented an open set of $\mathrm{Diff}^1(M)$ such that  every diffeomorphism in the open set has a $C^1$-robust homoclinc tangency. Nowadays,  it is called a \emph{$C^{1}$-Newhouse domain}. 

Before stating our result,
we recall the notion of classical regularity  which is   
quite opposite to that of historic behaviour.
We say that 
$f$ has a \emph{Dirac physical measure} $\nu$
associated with 
a wandering domain $\mathbb{D}$ 
 if   for every $x\in \mathbb{D}$
\[\lim_{n\to+\infty} \frac{1}{n+1}\sum_{i=0}^{n}\delta_{f^i(x)}=\nu\] 
and the support of $\nu$ is 
equal to a  periodic orbit of $f$. Moreover such a $\nu$ is \emph{non-trivial} 
if the periodic orbit is of saddle type.
It implies that 
$\mathbb{D}$ is contained in the basin of $\nu$, 
which  
has positive Lebesgue measure because $\mathbb{D}$  is a non-empty open set.
In this sense,  it is traditional usage to refer to $\nu$ as \emph{physical} or \emph{SRB}, see \cite{CTV19}. 
Dirac physical measures were studied for some transitive flows such that 
the supports of measures are non-attracting orbits in \cite{SSV10,SV13}. On the other hand, for 
diffeomorphisms,  using the prototype of wild hyperbolic set, 
Colli and Vargas presented a non-trivial Dirac physical measure associated with a wandering domain in \cite{CV01}, which can  be extended in 
some dense subset of the $C^{r}$($2\leq r<\infty$)-Newhouse domain in  \cite{KS17}. 
In a $C^{1}$-generic standpoint, several negative observations 
about the existence of Dirac physical measures supported on non-attracting periodic orbits and examples are provided in \cite{S18,GGS20}.

We are now ready to state the main theorem.
  \begin{mthm}\label{thm0}
  There is a 3-dimensional diffeomorphism $f$ in the $C^{1}$  Newhouse domain
  such that every $C^{r}$ $(1\leq r<\infty)$ neighbourhood of $f$ contains
  two types of diffeomorphisms, one of which has a historic contracting wandering domain, 
  and the others have non-trivial Dirac physical measures  
  supported by saddle periodic orbits.
 \end{mthm}

\begin{rmk}\label{novelty1}
The following scenario may come to mind to detect historic wandering domains using several known  facts
in $C^{1}$ dynamics of dimension at least 3. 
In fact, Barrientos  \cite{Ba-pre} essentially implemented the scenario. In \cite{BDP03} 
Bonatti, D\'{\i}az and Pujals gave $C^{1}$ open sets of diffeomorphisms  
with a $C^{1}$ dense subset
of diffeomorphisms admitting a periodic disk, that is, 
there is a large integer $n$ such that $f^{n}$ is equal to the identity map on some disk. 
Then perturb $f^{n}$ in order to get a $C^{2}$-homoclinic tangency 
in a smooth normally contracting surface $\Sigma$ by \cite{Ro95}, 
and apply \cite{KS17} to get a historic wandering domain
by perturbation of the restriction $f^{n}\vert_{\Sigma}$. 
However, the proofs in \cite{KS17} are founded on some \emph{unbalanced weight assumption} 
to control dynamics, see \cite[Remark 2.1]{KS17}.
Thus, 
one can use the scenario  to confirm the existence of non-trivial Dirac physical measure supported on the saddle fixed point,  but one cannot use it to confirm that on the 2-periodic orbit.
On the other hand,  
the direct way provided in this paper is 
not only useful for both, but may be applied to 
non-trivial Dirac physical measure supported by saddle periodic orbit of every period.
See the final Remark \ref{tg}.
\end{rmk}

We believe that there exist a locally dense $C^{1}$ diffeomorphisms
which satisfy the same properties as in Theorem \ref{thm0} and Remark \ref{tg}.
In particular, there will be Dirac physical measures for periodic orbits of every period.
While this paper is limited to specific models, 
it will shed light on ``pluripotentiality'' of wandering domains, 
that is, the existence of them whose orbits can approximate statistically to every dynamics on given  invariant sets.

Note that the residual subset of blender-horseshoe causing historic behaviour in \cite{BKNRS20} 
has zero Lebesgue measure, while the historic wandering domain given in Theorem \ref{thm0} has positive Lebesgue measure. We mention another novelty that this paper contains.
To prove Theorem \ref{thm0}, we borrow the key idea called  
``critical chain'' from \cite{CV01}.  However Colli and Vargas proved their main technical results (Linking and Linear Growth Lemmas) which support the proof of their Critical Chain Lemma 
by using $C^{2}$-robust homoclinic tangencies, and hence they did not employ contexts of $C^{1}$-robust ones. 
So, instead of their technical results, we present an innovation (Lemma \ref{l_zeta} with Proposition \ref{lp_Prop_5_1} in Section \ref{S.CC}), which takes advantage of 
the distinctive property of inverse dynamics of the cs-blender horseshoe. It should also be emphasised that our proof might be considerably simpler than that of Colli-Vargas.

\subsection{Centre stable blender-horseshoes}
In this subsection and the next, we give a concrete construction of $f$ in Theorem \ref{thm0}. 
All this can be extended to structures called blender-horseshoes, which can be  
defined for diffeomorphisms of all dimensions greater than or equal to three.
However, a 3-dimensional case contains all essential properties on blender-horseshoes, and therefore we  discuss them only in this case.

Let 
$\lambda_{ss}, \lambda_{cs0}, \lambda_{cs1}$ and $\lambda_{u}$ be real positive constants with 
\begin{equation}\label{ev}
\lambda_{ss}<\lambda_{cs0}<1/2<\lambda_{cs1}<1<\lambda_{cs0}+\lambda_{cs1},\  2<\lambda_{u}.
\end{equation}
Furthermore, we suppose that $\lambda_{cs0}$ is relatively small compared to $\lambda_{cs1}$ and $\lambda_{u}$ so that  
\begin{equation}\label{pdc}
\lambda_{cs0}\lambda_{cs1}\lambda_{u}^{2}<1,
\end{equation}
which corresponds to the partially dissipative condition for $3$-dimensional diffeomorphisms given below.
We first consider the 2-dimensional affine horseshoe map $F$ with 
\[
F(x,y)=\left\{\begin{array}{ll}
(\lambda_{u}x,\lambda_{ss}y)&\text{if}\ (x,y)\in [0,\lambda_{u}^{-1}]\times[0,1], \\[2pt]
(\lambda_{u}(1-x), 1-\lambda_{ss}y)&\text{if}\ (x,y)\in [1-\lambda_{u}^{-1}, 1]\times[0,1],
\end{array}
\right.
\]
and the iterated function system consisting of the pair of contracting 1-dimensional maps defined 
as, for $z\in [0,1]$, 
\begin{subequations} 
 \begin{equation}\label{IFS}
\zeta_{0}(z)=\lambda_{cs0}z, \ 
\zeta_{1}(z)=\lambda_{cs1}z+\beta,
\end{equation}
where $\beta=1-\lambda_{cs1}$. 
Let 
$\mathbb{B}$ be the unit cube $[0,1]^{3}$ and let 
$f: \mathbb{B}\rightarrow \mathbb{R}^{3}$ be a local diffeomorphism
satisfying the following conditions:
\begin{itemize}
\item $f\vert_{\mathbb{V}_{0}\cup \mathbb{V}_{1}}$ is 
the skew product $F\ltimes(\zeta_{0}, \zeta_{1})$ given by 
 \begin{equation}\label{blender}
f(x,y,z)=
\left\{\begin{array}{ll}
(\lambda_{u}x,  \lambda_{ss}y, \lambda_{cs0}z )& \text{if}\ (x,y,z)\in \mathbb{V}_{0},\\[2pt]
(\lambda_{u}(1-x),   1-\lambda_{ss}y, \lambda_{cs1}z+\beta ) &  \text{if}\  (x,y,z)\in \mathbb{V}_{1},
\end{array}
\right.
\end{equation}
where $\mathbb{V}_{0}= [0,\lambda_{u}^{-1}]\times[0,1]^{2}$ and 
$\mathbb{V}_{1}= [1-\lambda_{u}^{-1},1]\times[0,1]^{2}$.
\item 
For 
$\mathbb{G}=\mathbb{B}\setminus(\mathbb{V}_{0}\cup \mathbb{V}_{1})$, $f(\mathbb{G})$ is 
contained in $\mathbb{R}^3\setminus \mathbb{B}$.
\end{itemize}
\end{subequations} 
We now consider the hyperbolic set $\Lambda=\bigcap_{i\in \mathbb{Z}} f^{i}(\mathbb{V}_{0}\cup \mathbb{V}_{1})$ of $f$ in $\mathbb{B}$ 
on which $f$ is conjugate to the full shift of two symbols, and besides 
$\Lambda$ contains the saddle fixed points  
\[
P=(0,0,0)\in \mathbb V_{0}\cap f(\mathbb V_{0}),\ 
Q=\left(\lambda_{u}(1+\lambda_{u})^{-1}, (1+\lambda_{ss})^{-1}, 1\right)\in \mathbb V_{1}\cap f(\mathbb V_{1}).
\]
\begin{rmk}[Asymmetricity]\label{rmk2}
The inequalities 
\eqref{ev} and \eqref{pdc} imply a partially dissipative situation for $f\vert_{\Lambda}$, 
which gives  asymmetrical contractions along the centre-stable direction for the cs-blender horseshoe, see Figure \ref{fig/blender}.
We will see that these conditions are essential to show 
Lemma \ref{lem7.3} and Theorem \ref{thmWD}.
Note that these conditions  cannot be fulfilled if both $\zeta_{0}$ and $\zeta_{1}$ are close to the identity.
Therefore, 
the cs-blender horseshoe with partially dissipative situation
might be $C^{1}$-away from usual one which can be derived 
from a heterodimensional cycle via some strongly homoclinic intersection 
by an arbitrarily small perturbation in \cite[Section 4]{BD08}. 
See also Remark \ref{tg}.

\end{rmk}
\begin{figure}[hbt]
\centering
\scalebox{0.6}{
\includegraphics[clip]{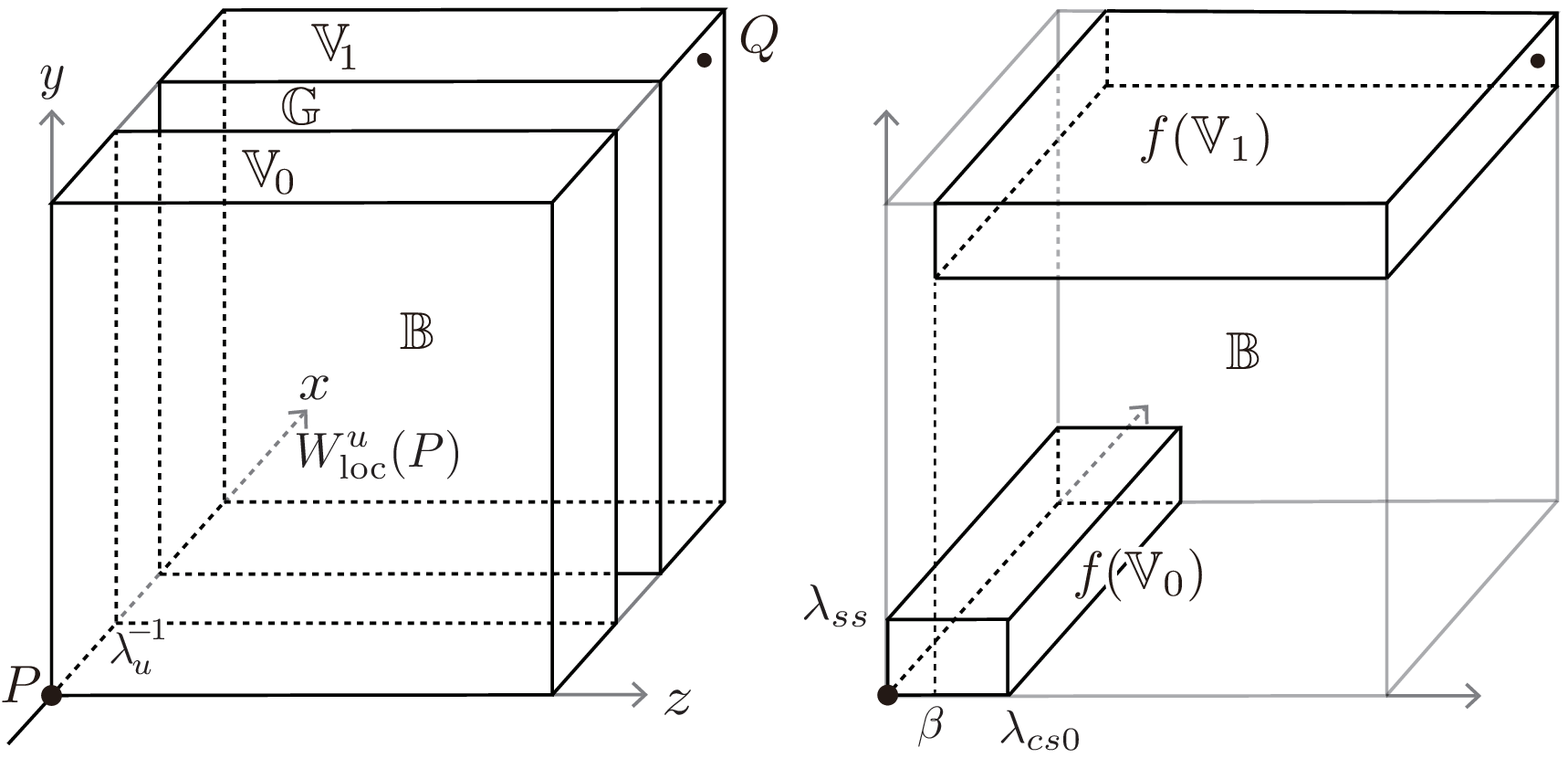}
}
\caption{ } 
\label{fig/blender}
\end{figure}

Despite its asymmetric structure, $\Lambda$ still satisfies all the standard properties of blender. 
In fact, 
$\Lambda$ is an example of \emph{cs-blender horseshoe},
refer to \cite[Definition 3.9]{BD12} for the precise definition.
Note that 
there is a $C^{1}$-neighbourhood $\mathcal{N}_{f}$ of $f$  
such that every $\tilde{f}\in \mathcal{N}_{f}$ 
has the continuation $\Lambda_{\tilde{f}}$ of $\Lambda$ which is 
a cs-blender horseshoe containing the 
continuations $P_{\tilde{f}}$ and $Q_{\tilde{f}}$ 
of $P$ and $Q$, respectively.
Moreover it follows from \eqref{ev} that $\beta<\lambda_{cs0}$, 
that is, $\tilde f$  still  has a superposition region associated with $\Lambda_{\tilde f}$ and lying between $\lumfd(P_{\tilde{f}})$ and $\lumfd(Q_{\tilde{f}})$. See
\cite[Lemma 1.11]{BD96} and  \cite[Lemma 3.10]{BD12} for details.

\subsection{Configurations of tangency}
Next, to investigate homoclinic tangencies under the setting of cs-blender horseshoe,
we assume the following conditions on the second iterate of $f\vert_{\mathbb{G}}$:  for  a given $0<\delta<\frac{1}{2}-\lambda^{-1}_{u}$, the restriction of $f^{2}$ to the $\delta$-neighbourhood $U_{\delta}$ of the 
2-dimensional disc $\{x=1/2\}\cap \mathbb{G}$ is given by   
\begin{subequations} 
\begin{equation}\label{tang}
f^{2}(x,y,z)=\left(-a_{1}\Bigl(x-\frac{1}{2}\Bigr)^{2}+a_{2}z,\ a_{3}\Bigl(y-\frac{1}{2}\Bigr)+\frac{1}{2},\ 
a_{4}\Bigl(x-\frac{1}{2}\Bigr)+\frac{1}{2} \right)
\end{equation}
for $(x,y,z)\in U_{\delta}$, where the coefficients $a_{1}$, $a_{2}$, $a_{3}$ and $a_{4}$ are nonzero
real constants with 
\begin{equation}\label{a_i}
a_{1}>(1-2\lambda_{u}^{-1})^{-1},\  
a_{2}, a_{4}>0,\ 
|a_{3}|<1-2\lambda_{ss}.
\end{equation}
\end{subequations} 
The first condition is used to show Proposition \ref{proj-w} which 
leads to wandering domains disjoint from $\Lambda$, and
the second one is necessary to show Lemma \ref{rt}.
The last one assures that $f^{2}(\mathbb{G})\cap \mathbb{B}$
has no intersection with $f(\mathbb{V}_{0})\cup f(\mathbb{V}_{1})$, see Figure \ref{tang_fig0}.

 \begin{figure}[hbt]
\centering
\scalebox{0.6}{
\includegraphics[clip]{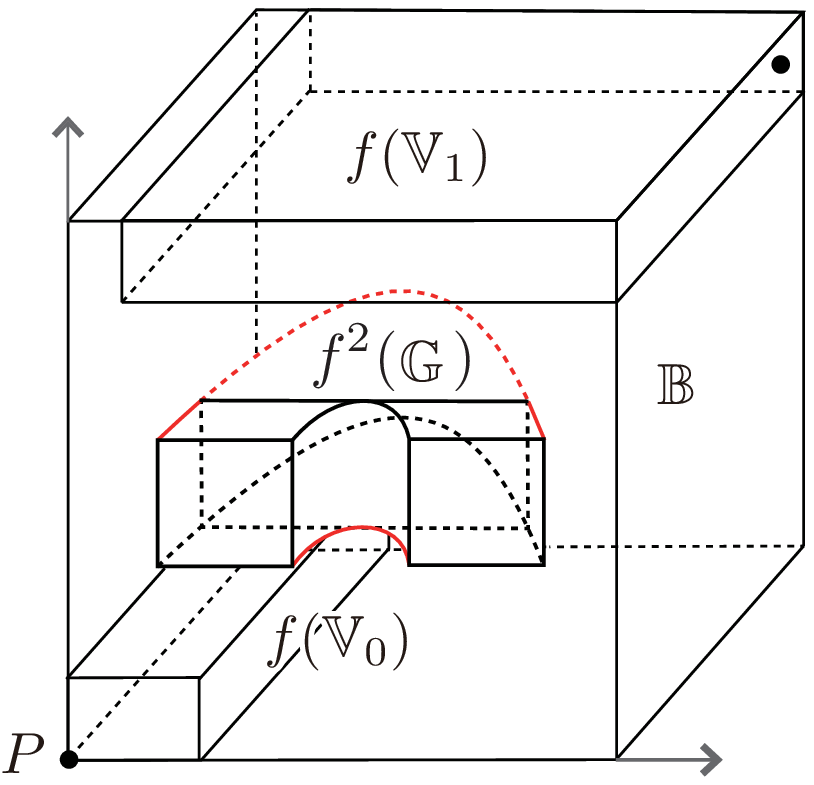}
}
\caption{} 
\label{tang_fig0}
\end{figure}

We say that 
a blender horseshoe $\Lambda$ is \emph{wild} if 
there are points $x, x^{\prime}\in \Lambda$ 
such that $\umfd(x)$ and $\smfd(x^{\prime})$ have 
a non-transverse intersection.
Let 
the $f^{2}$-image of $(1/2,0,0)\in \mathbb{G}$ be written as $X$, 
which satisfies  
\[X=(0, 1/2-a_{3}/2, 1/2)\in \umfd(P)\cap\lsmfd(P),\  
T_{X}\umfd(P)\subsetneq T_{X}\lsmfd(P).\] 
That is, $f$ has a homoclinic tangency of $P\in \Lambda$.  See Figure \ref{tang_0_fig}-(a). 
Therefore the cs-blender horseshoe $\Lambda$ for \eqref{blender} is wild. 
Furthermore, this homoclinic tangency is robust for small $C^{1}$ perturbation:
\begin{lem}\label{rt}
Let $f$ be a diffeomorphism with \eqref{blender} and \eqref{tang}. 
Then $f$  has a $C^{1}$-robust homoclinic tangency of the cs-blender horseshoe $\Lambda$.
\end{lem}
\noindent 
The proof of this lemma is given in the Appendix \ref{apdx} as it is just comfirmed that 
\cite[Theorem 4.8]{BD12} can be applicable to our setting.

 \begin{figure}[hbt]
\centering
\scalebox{0.6}{
\includegraphics[clip]{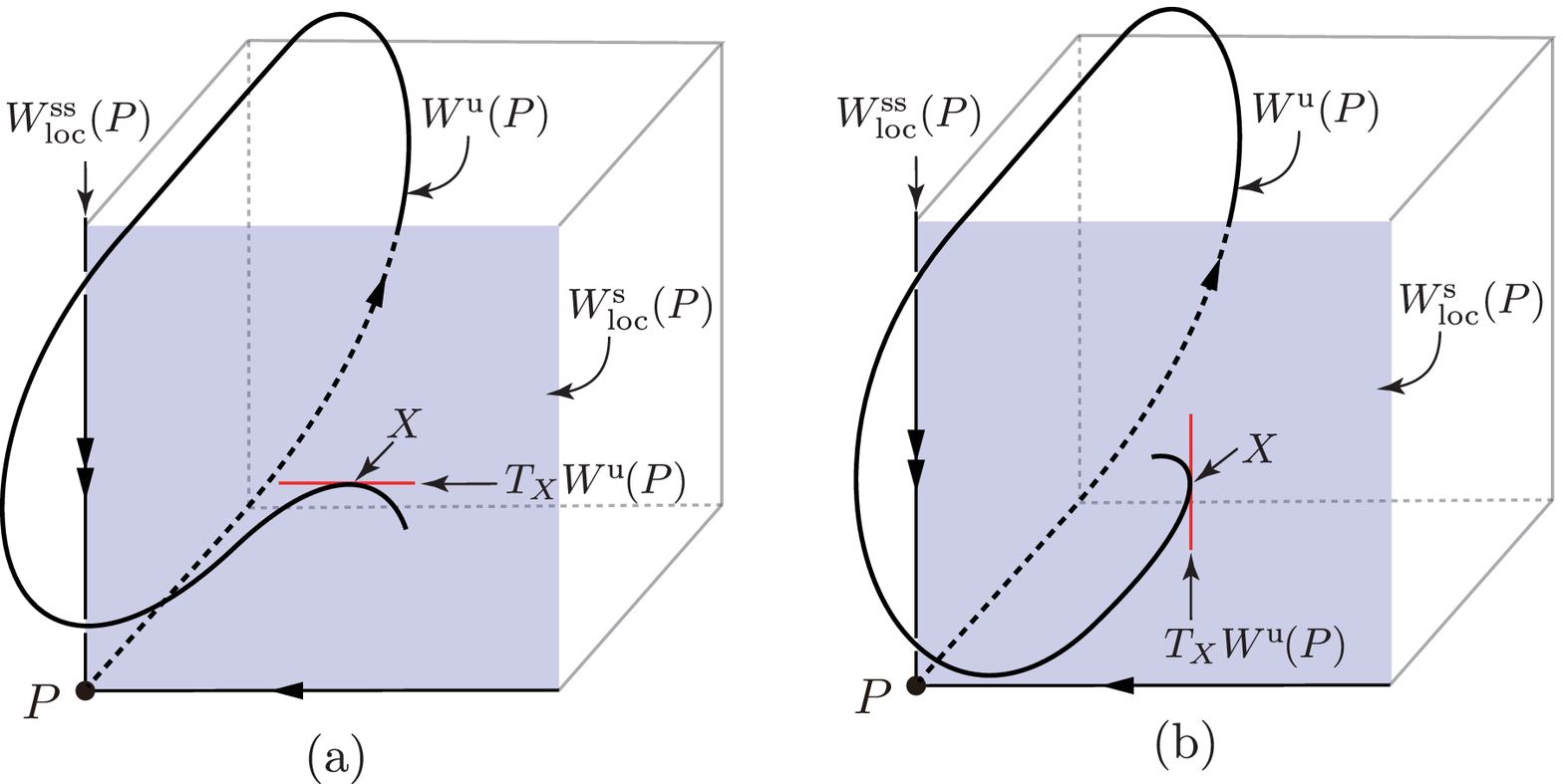}
}
\caption{} 
\label{tang_0_fig}
\end{figure}

\begin{rmk}[Generality of configuration]
We here  explain why the above setting is  a non-trivial extension of \cite{CV01}, 
which is also the second novelty of the present paper.
For a certain map instead of \eqref{tang}, 
 we have another situation such that the direction of 
 $T_{X}\umfd(P)$ is parallel to the ss-direction of $\Lambda$, as in Figure \ref{tang_0_fig}-(b). 
This is an actually trivial extension of  Colli-Vargas' 2-dimensional model with one more dimension, and hence one might obtain similar results as in \cite{CV01} with the help of several techniques in \cite{PV94}. 
However, such a strategy will not get us out of the  $C^{2}$-category. 
Furthermore, 
with a little perturbation, it is possible to maintain the tangency but make a direction different from the ss-direction. Thus, the tangent directions at forward images of such a perturbed tangency are gradually close to the cs-direction, since it is pressed strongly along the ss-direction. 
As a consequence, the situation will be essentially the same as defined by \eqref{tang} 
as in Figure \ref{tang_0_fig}-(a). 
In this sense, the configuration of tangency in this paper is general. 
Combining this situation with Lemma \ref{rt}, we now obtain 
a diffeomorphism having both cs-blender horseshoe and $C^{1}$-robust homoclinic tangency simultaneously.
 \end{rmk}

It follows from  the above definitions and facts that Theorem \ref{thm0} is a consequence of the next theorem.
  \begin{subthm}\label{thm1}
  Every $C^{r}$$(1\leq r<\infty)$-neighbourhood of the above diffeomorphism $f$ with a wild blender-horseshoe contains a diffeomorphism 
 which has a historic 
contracting
 wandering domain. 
 Moreover, it contains 
 another diffeomorphism having non-trivial Dirac physical measures supported by saddle periodic orbits associated with a contracting
 wandering domain.
\end{subthm}

For the proof of Theorem \ref{thm1}, 
we need to prepare some 
tools associated with the blender-horseshoes,  and give
key results (Lemma \ref{l_zeta}, Proposition \ref{lp_Prop_5_1}) 
for projected dynamics in Section \ref{S.CC} and  
some infinite sequence of perturbations in Section \ref{S.pertb}.
Using the results,   
the existence of the wandering domain is proved by several geometric steps in 
Theorem \ref{thmWD} of Section \ref{s7}.
Finally, 
the existences of historic behaviour (Theorem \ref{hist}) and Dirac physical measure (Theorem \ref{phys}) are shown by probabilistic approaches in Section \ref{s8}.
\smallskip 

\section{Critical chains of bridges}\label{S.CC}
The results given in this section are keys to this paper,
which is associated with several subsequences of u-bridges and its copies in the cs-direction.  
\subsection{Unstable bridges and gaps}\label{ss4.1}
Let $f$ 
be a diffeomorphism with the cs-blender horseshoe $\Lambda$ by 
\eqref{blender}.
We first 
extend the notations of bridges and gaps given in \cite{CV01} as follows.
For any integer $n\geq 1$, let  
$\underline{w}$ be an $n$-tuple of binary codes, that is, $\underline{w}=w_{1}\cdots w_{n}$ with $w_{i}\in \{0,1\}$. 
Define dynamically defined rectangular solids as 
\begin{align*}
\mathbb{B}^{\rm u}(n; \underline{w})
&=\left\{x\in \mathbb{B}\ :\  f^{i-1}(x)\in \mathbb{V}_{w_{i}},
i=1,\ldots,n\right\},  \\
\mathbb{G}^{\rm u}(n; \underline{w})
&=\mathbb{B}^{\rm u}(n; \underline{w})\setminus 
\left(\mathbb{B}^{\rm u}(n+1; \underline{w}0)\cup \mathbb{B}^{\rm u}(n+1; \underline{w}1)\right).  
\end{align*}
The former set is called an \emph{unstable bridge} or 
a \emph{u-bridge}, while the latter one an  
\emph{unstable gap} or  a \emph{u-gap}.
Sometimes $n$ and $\underline{w}$ of $B^{\rm u}(n; \underline{w})$ are called the \emph{generation} and \emph{itinerary} 
for the u-bridge, respectively.
If there is no confusion, the number of generation 
may be omitted and $\mathbb{B}^{\rm u}(n; \underline{w})$ and 
$\mathbb{G}^{\rm u}(n; \underline{w})$ may be written as $\mathbb{B}^{\rm u}(\underline{w})$ 
and $\mathbb{G}^{\rm u}(\underline{w})$, respectively.
Observe that if $n$ is fixed, the family
$\{ 
\mathbb{B}^{\rm u}(\underline{w}) : \underline{w}\in \{0,1\}^{n}\}$ 
 consists of 
$2^{n}$ mutually disjoint rectangular solids, which consequently contains 
$2^{n}$ mutually disjoint arcs of $\lumfd(P)$. 
For every $n\geq 1$ and $\underline{w}\in \{0,1\}^{n}$, 
we denote by 
$B^{\rm u}(n; \underline{w})$ (or $B^{\rm u}(\underline{w})$ for short)
the arc $\mathbb{B}^{\rm u}(n; \underline{w})\cap \lumfd(P)$, which 
can be regarded as a subinterval in $[0,1]$, that is, 
\[\mathbb{B}^{\rm u}(\underline{w})\cap \lumfd(P)=B^{\rm u}(\underline{w})\times\{(0,0)\}.\]
Since 
$\mathbb{G}^{\rm u}( \underline{w})\subset \mathbb{B}^{\rm u}(\underline{w})$, 
one can obtain the open interval $G^{\rm u}(n; \underline{w})$ (or $G^{\rm u}(\underline{w})$ for short) on $[0,1]$ satisfying
\[\mathbb{G}^{\rm u}(\underline{w})\cap \lumfd(P)=G^{\rm u}(\underline{w})\times\{(0,0)\}.\]
The closed interval $B^{\rm u}(\underline{w})$  is called  a \emph{u-bridge}, while the open interval 
$G^{\rm u}(\underline{w})$ is called a \emph{u-gap} of  
the \emph{u-Cantor set}
$\Lambda_{u}=\Lambda\cap \lumfd(P)$.  
Finally,  we write
$G_{0}^{\rm u}=[0, 1]\setminus(B^{\rm u}(0)\cup B^{\rm u}(1))$,
and hence  $\mathbb{G}\cap \lumfd(P)=
G_{0}^{\rm u}\times\{(0,0)\}$.

\subsection{Projected dynamics}\label{ss.pd}
The following simple projection can be used,
since our model consists of the affine forms by \eqref{blender}
with a tangency without any distortion given in \eqref{tang}.
For any $(x,y,z)\in \mathbb{B}$ and integer $n>0$, we write
\begin{equation}\label{varphi}
\varphi^n(x,z)=\widehat\pi(f^n(x,y,z))
\end{equation}
if the value of the right-hand side of the equation does not depend on $y$, 
where $\widehat\pi:\mathbb{B}\longrightarrow \mathbb{R}^2$ is the projection defined by $\widehat\pi(x,y,z)=(x,z)$.
By \eqref{tang}, we have
\begin{equation}\label{eqn_1/2z}
\varphi^2(1/2,z)=(a_2z,1/2).
\end{equation}

First, we define sequences 
$\{B_{k}^{\rm u}\}_{k\geq 1}$ and $\{\widetilde B_{k}^{\rm u}\}_{k\geq 0}$
of unstable bridges as follows.
For any integer $n_{0}\geq 0$ and any code $\underline{\widetilde w}^{(0)}\in \{0,1\}^{n_{0}}$, 
let us define
\[\widetilde B_0^{\rm u}=B^{\rm u}(n_0,\underline{\widetilde w}^{(0)}),\] 
which is a $u$-bridge 
contained in $\pi_1\circ \varphi^2(\{1/2\}\times [0,1])$. 
One can take $n_{0}$ so that 
 $\widetilde B_0^{\rm u}\subset (0, a_{2})$, see Figure \ref{f_BkBk}.
Let $B_1^{\rm u}$, $\widetilde B_1^{\rm u}$ be the pair of maximal sub-bridges of $\widetilde B_0$ such that 
$\widetilde B_1^{\rm u}$ lies in the left side of $B_1^{\rm u}$, that is, 
$\max \widetilde B_1^{\rm u}<\min B_1^{\rm u}$.
Then they are represented as 
$B_1^{\rm u}=B^{\rm u}(n_0+1,\underline{\widetilde w}^{(0)}\alpha_1)$ and 
$\widetilde B_1^{\rm u}=B^{\rm u}(n_0+1,\underline{\widetilde w}^{(0)}\widetilde\alpha_1),
$ where $\alpha_1$ is either 0 or 1 and $\widetilde\alpha_1=1-\alpha_1$.
If we write $\underline{\widetilde w}^{(0)}\alpha_1=\underline{w}^{(1)}$ and 
$\underline{\widetilde w}^{(0)}\widetilde\alpha_1=\underline{\widetilde w}^{(1)}$, 
then 
\[B_1^{\rm u}=B^{\rm u}(n_0+1,\underline{w}^{(1)}),\ 
\widetilde B_1^{\rm u}=B^{\rm u}(n_0+1,\underline{\widetilde w}^{(1)}).
\] 
For integer $k>1$, 
we inductively define
the sub-bridges
$B_k^{\rm u}$ and $\widetilde B_k^{\rm u}$ of $\widetilde B_{k-1}^{\rm u}$ 
satisfying 
\[
B_k^{\rm u}=B^{\rm u}(n_0+k,\underline{w}^{(k)}),\ 
\widetilde B_k^{\rm u}=B^{\rm u}(n_0+k,\underline{\widetilde w}^{(k)}), 
\]
where $\underline{w}^{(k)}=\underline{\widetilde w}^{(k-1)}\alpha_k$ and $\underline{\widetilde w}^{(k)}=\underline{\widetilde w}^{(k-1)}\widetilde\alpha_k$ 
for some $\alpha_k,\widetilde\alpha_k$ with $\{\alpha_k,\widetilde\alpha_k\}=\{0,1\}$.
See Figure \ref{f_BkBk}.
\begin{figure}[hbtp]
\centering
\scalebox{0.5}{\includegraphics[clip]{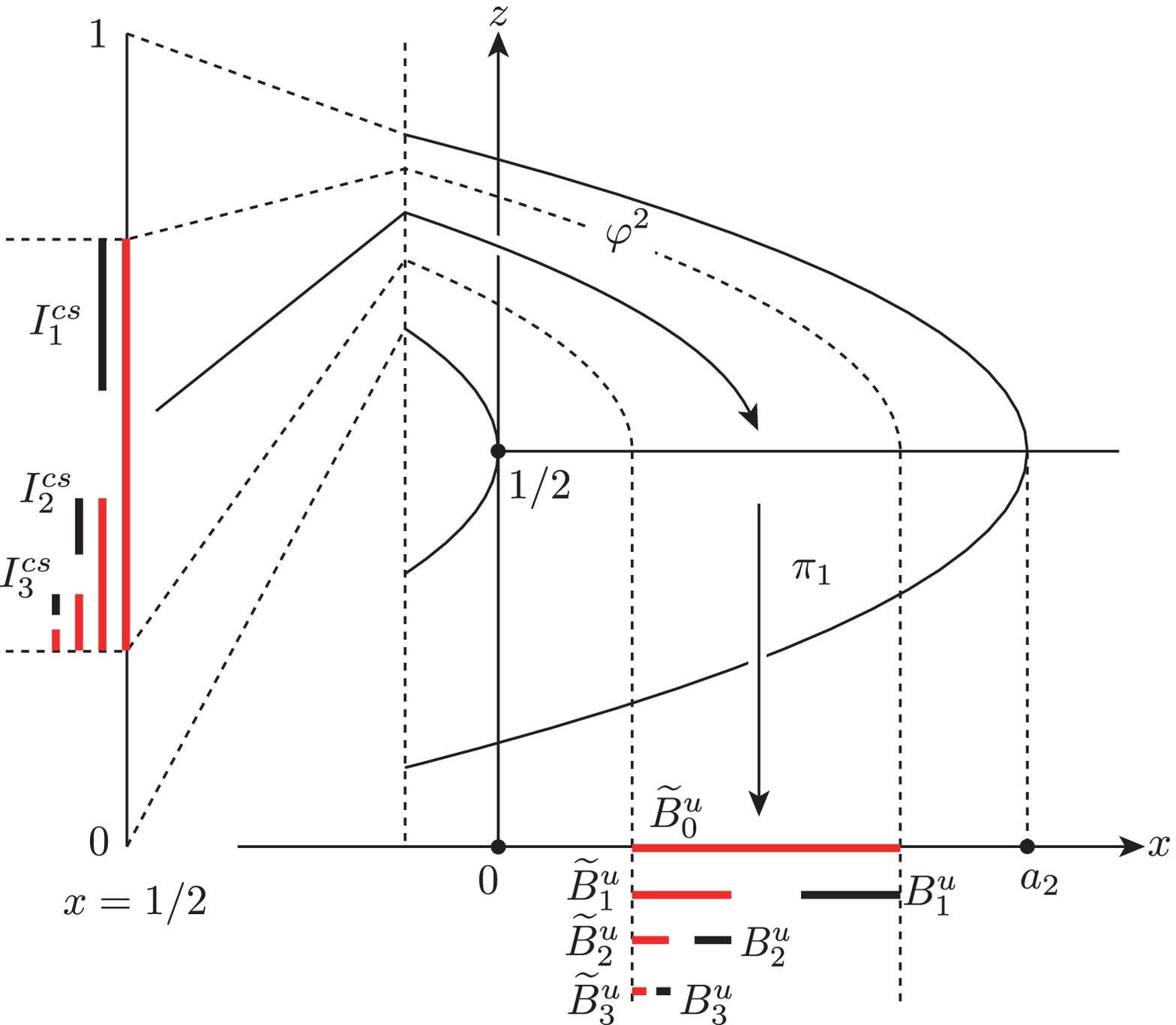}}
\caption{}
\label{f_BkBk}
\end{figure}

Next, we define a sequence
$\{J_{k}^{\rm cs}\}_{k\geq 1}$ 
of $\varphi^{2}$-inverse images of $B^{\rm u}_{k}$ as follows.
For every integer $k\geq 1$ and 
sub-bridge $B^{\rm u}_{k}$ of $\widetilde B_0^{\rm u}$, let $I_{k}^{\rm cs}$ be the arc in 
$\{1/2\}\times [0,1]$ 
with 
\[\pi_1\circ\varphi^2(I^{\rm cs}_{k})=B^{\rm u}_{k},\] 
where $\pi_{1}:\mathbb{R}^{2}\to \mathbb{R}$ is the projection with $\pi_{1}(x,z)=x$.
Define $J_{k}^{\rm cs}$ as the sub-interval of $[0,1]$ such that 
\[\{1/2\}\times J_{k}^{\rm cs}=I_{k}^{\rm cs},\]
 and 
call it the \emph{cs-interval} associated with $B^{\rm u}_{k}$. 
By \eqref{eqn_1/2z}, 
\begin{equation}\label{J&B}
J^{\rm cs}_{k}=a_2^{-1} B^{\rm u}_{k}.
\end{equation}

For any code $\underline{\gamma}=\gamma_1\gamma_2\cdots \gamma_n\in \{0,1\}^{n}$, 
the map $\boldsymbol{\zeta}_{\underline{\gamma}}^n$ 
(or $\boldsymbol{\zeta}_{\underline{\gamma}}$ for short) is defined by
\[
\boldsymbol{\zeta}_{\underline{\gamma}}^n=\zeta_{\gamma_n}\circ \cdots\circ \zeta_{\gamma_2}\circ \zeta_{\gamma_1},
\]
where each $\zeta_{\gamma_i}$ is the function given in \eqref{IFS}.
Moreover, we define the \emph{length} $\left|\underline{\gamma}\right|$ of $\underline{\gamma}$  
as the total number of symbols in $\underline{\gamma}$, that is,  
$\left|\underline{\gamma}\right|=\left|\gamma_1\gamma_2\cdots \gamma_n\right|=n$.

\begin{lem}\label{l_zeta}
For any integer $L>0$, 
any u-bridge
$\bar  B_{k}^{\rm u}$ $(k=1,2,\ldots)$  with 
$\bar B_{k}^{\rm u}
 = B^{\rm u}(n_0+k+Lk, \underline{w}^{(k+Lk)})\subset B_{k}^{\rm u}$
and any code $\underline{u}^{(k)}$ with $|\underline{u}^{(k)}|\geq 0$, 
there exist codes $\underline{\widehat w}^{(k)}$ 
and 
sub-bridges 
$\widehat{B}^{\rm u}_{k}=B^{\rm u}(\widehat n_{k},\underline{\widehat w}^{(k)})$ of $\bar B^{\rm u}_{k}$
satisfying the following conditions.
\begin{enumerate}[\rm (1)]
\item
$\underline{\widehat w}^{(k)}=\underline{w}^{(k+Lk)}\underline{v}^{(k)}\underline{\gamma}^{(k)}$, where
\begin{itemize}\setl
\item
$\underline{\gamma}^{(k)}=\gamma_{m_k}\gamma_{m_k-1}\cdots \gamma_2\gamma_1$, 
 $\gamma_i\in \{0,1\}$ $(i=1,2,\dots,m_k)$, 
for some integer $m_k$  satisfying 
$0<m_k\leq N_0+N_1k$, 
where $N_0$, $N_1$ are positive integers independent of $k$,
\item
either $\underline{v}^{(k)}=\underline{u}^{(k)}$ or $\underline{v}^{(k)}=\underline{u}^{(k)}\alpha$ for some 
$\alpha\in \{0,1\}$.
\end{itemize}
\item
$\boldsymbol{\zeta}_{\underline{\widehat w}^{(k)}}(1/2)\in 
\bar J_{k+1}^{\rm cs}$ and 
$\widehat{J}^{\rm cs}_{k+1}\subset \bar J_{k+1}^{\rm cs}$, 
where $\bar J^{\rm cs}_{k+1}$ is the cs-interval associated with $\bar B_{k+1}^{\rm u}$.
\end{enumerate}
\end{lem}

\begin{rmk}\label{rmk_freedom}
The freedom of the choice of $\underline{u}^{(k)}$ in Lemma \ref{l_zeta} is crucial in 
the study of historic behaviour of wandering domains for diffeomorphisms $C^1$-close to $f$.
\end{rmk}

\begin{proof}[Proof of Lemma \ref{l_zeta}]
Take $\gamma_1\in \{0,1\}$ with $\bar J_{k+1}^{\rm cs}\subsetneq \Im\,(\zeta_{\gamma_1})$, 
where $\Im(\zeta_{\gamma_1})$ is the image of $\zeta_{\gamma_1}$. 
Define 
\[\bar J_{k+1}^{\rm cs(1)}=\zeta_{\gamma_1}^{-1}(\bar J_{k+1}^{\rm cs}).\]
For any integer $i\geq 1$, 
 define the interval $\bar J_{k+1}^{{\rm cs}(i)}$ in $[0,1]$ inductively if 
$\bar J_{k+1}^{{\rm cs}(i-1)}\subset \Im\,(\zeta_{\gamma_{i}})$ holds for at least one $\gamma_i$ of $0,1$.
Suppose that this process finishes $m_k$ times.
Accordingly, $\bar J_{k+1}^{{\rm cs}(m_k-1)}$ is contained in $\Im\,(\zeta_{\gamma_{m_k}})$. 
We here 
note  that 
\[
\left| \bar B_{k+1}^{\rm u}\right|=
(\lambda_{u}^{-1})^{n_{0}+k+1+L(k+1)}
\]
and 
\[\pi_1\circ\varphi^2(
\{1/2\}\times \bar J_{k+1}^{\rm cs}
)=\bar B^{\rm u}_{k+1}.\] 
Then, by \eqref{eqn_1/2z}, 
\[
|\bar J_{k+1}^{\rm cs}|=|a_{2}^{-1}|(\lambda_{u}^{-1})^{n_{0}+k+1+L(k+1)}.
\]
Thus,  
there are integers $m_{k-1}^{(0)},m_{k-1}^{(1)}$ with 
$m_{k}-1=m_{k-1}^{(0)}+m_{k-1}^{(1)}$ such that 
\[
|\bar J_{k+1}^{{\rm cs}(m_k-1)}|=
\lambda_{cs0}^{-m_{k-1}^{(0)}}\lambda_{cs1}^{-m_{k-1}^{(1)}}
(|a_{2}^{-1}|(\lambda_{u}^{-1})^{n_{0}+k+1+L(k+1)})
\leq 1-\beta,
\]
where $\lambda_{cs0}$ and  $\lambda_{cs1}$ are derivatives of $\zeta_{0}$ and $\zeta_{1}$, respectively, 
see \eqref{IFS}.
Since $\lambda_{cs0}<\lambda_{cs1}<1$, 
we have
\[m_k\leq 
\frac{\log |a_{2}|(1-\beta)\lambda_{u}^{n_{0}+1+L}}{\log \lambda_{cs0}^{-1}}+1
+\frac{\log\lambda_{u}^{1+L}}{\log \lambda_{cs0}^{-1}}k.
\]
Thus the smallest integers $N_0$ and $N_1$ with  
\[N_0\geq \frac{\log |a_{2}|(1-\beta)\lambda_{u}^{n_{0}+1+L}}{\log \lambda_{cs0}^{-1}}+1,\  
N_1\geq \dfrac{\log \lambda_u^{1+L}}{\log \lambda_{cs0}^{-1}}
\] 
fulfill the required condition on $m_k$.

From the definition of $m_k$, neither $\bar J_{k+1}^{{\rm cs}(m_k)}\subsetneq \Im\,(\zeta_0)=[0,\lambda_{cs0}]$ 
nor $\bar J_{k+1}^{{\rm cs}(m_k)}\subsetneq \Im\,(\zeta_1)=[\beta,1]$ occurs.
It follows that 
\[
\max\left\{\bar J_{k+1}^{{\rm cs}(m_k)}\right \}\geq \lambda_{cs0},\ 
\min\left\{\bar J_{k+1}^{{\rm cs}(m_k)}\right \}\leq \beta.
\]
So $\bar J_{k+1}^{{\rm cs}(m_k)}$ contains the interval $[\beta, \lambda_{cs0}]$.
In the case when $\boldsymbol{\zeta}_{\underline{w}^{(k+Lk)}\,\underline{u}^{(k)}}(1/2)\in [\beta, \lambda_{cs0}]$,  
we set $\underline{v}^{(k)}=\underline{u}^{(k)}$.
When $\boldsymbol{\zeta}_{\underline{w}^{(k+Lk)}\underline{u}^{(k)}}(1/2)\in [0,\beta)$ 
(resp.\ $\boldsymbol{\zeta}_{\underline{w}^{(k+Lk)}\underline{u}^{(k)}}(1/2)\in (\lambda_{cs0},1]$\,), 
we set $\underline{v}^{(k)}=\underline{u}^{(k)}1$ (resp.\ $\underline{v}^{(k)}=\underline{u}^{(k)}0$).
Since $1/2<\lambda_{cs}<1$, we have in either case 
$\boldsymbol{\zeta}_{\underline{w}^{(k+Lk)}{\underline v}^{(k)}}(1/2)\in [\beta,\lambda_{cs0}]$.
Hence the code 
\[\underline{\widehat w}^{(k)}=\underline{w}^{(k+Lk)}\underline{v}^{(k)}\gamma_{m_k}\gamma_{m_k-1}\cdots \gamma_2\gamma_1
\] satisfies 
our desired conditions.
\end{proof}

\begin{prop}\label{lp_Prop_5_1}
For any integer $L>0$,  let $\widehat B_k^{\rm u}=B^{\rm u}(\widehat n_k,\underline{\widehat w}^{(k)})$  $(k=1,2,\ldots)$  
be the sub-bridges of 
$\bar B_{k}^{\rm u}
= B^{\rm u}(n_0+k+Lk, \underline{w}^{(k+Lk)})
$ 
 and 
$\widehat J_{k+1}^{\rm cs}$ the cs-interval associated with $\widehat B_{k+1}^{\rm u}$ 
given in Lemma \ref{l_zeta}.
Then 
there exists a $t_{k+1}\in \mathbb{R}$ such that 
\begin{itemize}
\item
$\varphi^{\widehat n_k}(\widehat x_k^{\rm u},1/2)
=(1/2, \widehat z_{k+1}^{\rm cs})-(0,a_{2}^{-1}t_{k+1})
$,  
where $\widehat x_k^{\rm u}$ and $\widehat z_{k+1}^{\rm cs}$ are the centres of 
$\widehat B_k^{\rm u}$ and $\widehat J_{k+1}^{\rm cs}$, respectively,
\item $|t_{k+1}|< 
\lambda_{u}^{-(n_{0}+k+1+L(k+1))}$.
\end{itemize}
\end{prop}

\begin{rmk}\label{B&hatB}
{\rm
Note that $\widehat B_k^{\rm u} \subset \bar B_{k}^{\rm u}\subset B_{k}^{\rm u}$,  and  $B_{k}^{\rm u}$ will be used to specify the domain  of a perturbation. 
On the other hand, $\widehat B_k^{\rm u}$ determined from $\bar B_{k}^{\rm u}$ controls the size of the perturbation, and it will be important in the proof of  Proposition \ref{pbn}
that its size, which is exactly $|t_{k+1}|$ above, can be much smaller than the size of $B_{k}^{\rm u}$ by taking a sufficiently large $L$.
}
\end{rmk}

\begin{proof}[Proof of Proposition \ref{lp_Prop_5_1}]
Since $|\underline{\widehat w}^{(k)}|=\widehat n_k$, we have $\pi_1\circ \varphi^{\widehat n_k}(\widehat x_k^{\rm u},1/2)=1/2$ 
and hence, by Lemma \ref{l_zeta}, 
there is the cs-interval
$\bar  J_{k}^{\rm cs}\subset J_{k}^{\rm cs}$ and 
\[\varphi^{\widehat n_k}(\widehat x_k^{\rm u},1/2)
=\left(1/2, \boldsymbol{\zeta}_{\underline{\widehat w}^{(k)}}(1/2) \right)
\in \{1/2\}\times \bar J_{k+1}^{\rm cs}=\bar I_{k+1}^{\rm cs}.\]
See Figure \ref{f_BkBk2}. 
\begin{figure}[hbtp]
\centering
\scalebox{0.5}{\includegraphics[clip]{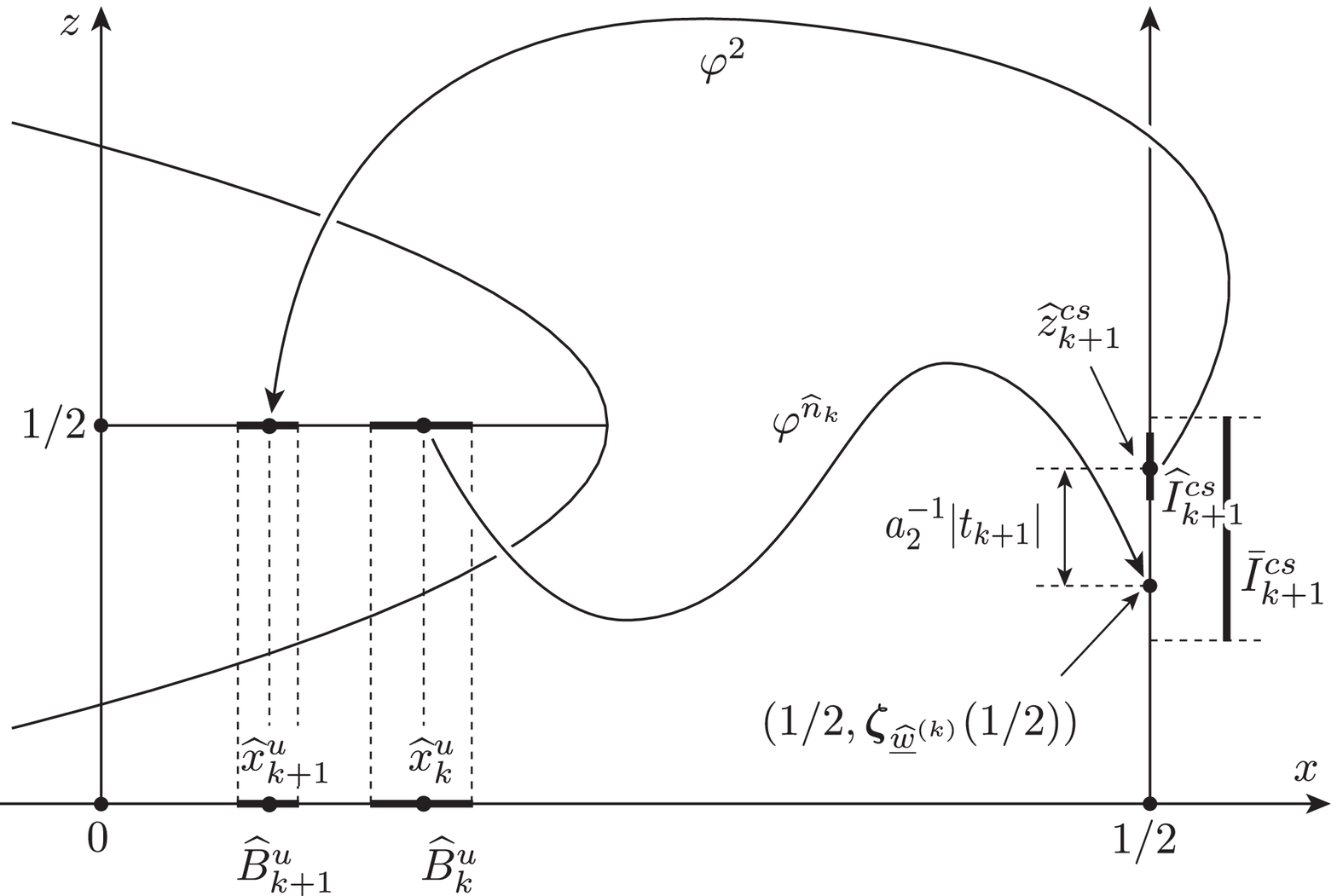}}
\caption{}
\label{f_BkBk2}
\end{figure}
Furthermore, 
 \[\widehat z_{k+1}^{\rm cs}\in \mathrm{Int} \widehat J_{k+1}^{\rm cs}\subset \bar J_{k+1}^{\rm cs}.\]
We here set 
\begin{equation}\label{t_{k+1}}
t_{k+1}=
a_{2}(
\widehat z_{k+1}^{\rm cs}-
\boldsymbol{\zeta}_{\underline{\widehat w}^{(k)}}(1/2)
).
\end{equation}
Since both $\boldsymbol{\zeta}_{\underline{\widehat w}^{(k)}}(1/2)$ and $\widehat z_{k+1}^{\rm cs}$ 
belong to $\bar J_{k+1}^{\rm cs}$, 
it follows from \eqref{J&B} that
\[|t_{k+1}|\leq a_{2}|\bar J_{k+1}^{\rm cs}|=
|\bar B_{k+1}^{\rm u}|=
\lambda_{u}^{-(n_{0}+k+1+L(k+1))}.
\qedhere
\]
\end{proof}

\section{Perturbations}\label{S.pertb}
We construct some map arbitrarily $C^{r}$-close to $f$ by countably many small perturbations
near homoclinic tangencies.

For any integer $k\geq 1$ let
$B_k^{\rm u}=B^{\rm u}(n_0+k,\underline{w}^{(k)})$ be the u-bridge  
and 
$J_{k}^{\rm cs}$  the cs-interval associated with $B_{k}^{\rm u}$ defined in the previous section. 

\begin{prop}\label{pbn}
For any $\varepsilon>0$, there is 
a diffeomorphism 
$g$ which is contained in 
the $\varepsilon$-neighbourhood of $f$ in the $C^{r}$-topology $(1\leq r <\infty)$ and 
satisfies the following conditions:
\begin{itemize}
\item 
$g\vert_{U_{3\delta/2}^{c}}=f$ for any 
$0<\delta< \frac{1}{3}(1-2\lambda^{-1}_{u})$, where $U_{3\delta/2}$ is the $3\delta/2$-neighbourhood  of the 
$2$-dimensional disc $\{x=1/2\}\cap \mathbb{B}$ and $U_{3\delta/2}^{c}$ is the complement of 
$U_{3\delta/2}$ in $\mathbb{B}$.
\item For every  $k\geq 1$ and  
$(x,y,z)\in 
[\frac{1}{2}-\delta, \frac{1}{2}+\delta]\times
[0,1]
\times J_{k+1}^{\rm cs}$, 
\[
g^{2}(x,y,z)=(t_{k+1}, 0, 0)+f^{2}(x,y,z),
\]
where $t_{k+1}$ is the number given in Proposition \ref{lp_Prop_5_1}.
\end{itemize}
\end{prop}

\begin{proof}
The idea of the proof is already described in Remark \ref{B&hatB}. 
Here we prove it in practice.

Let $\boldsymbol{b} :\mathbb{R}\longrightarrow \mathbb{R}$ be  
a non-negative, non-decreasing $C^{r}$ function 
such that $\boldsymbol{b} (x)=0$ if $x\leq -1$ while $\boldsymbol{b}(x)=1$ if $x\geq 0$. 
Using $\boldsymbol{b} $, we consider the bump function $\boldsymbol{b} _{\rho, I}$ with 
$\boldsymbol{b} _{\rho, I}=1$ on $I$ as follows: 
\[
\boldsymbol{b} _{\rho, I}(x)=\boldsymbol{b} \biggl(\frac{x-a}{\rho|I|}\biggr)+\boldsymbol{b} \biggl(-\frac{x-b}{\rho|I|}\biggr)-1,
\]
where $\rho$ is a positive constant and $I$ is the interval $[a,b]$ with $a<b$. 
The function satisfies 
\[
\left\|
\boldsymbol{b} _{\rho, I}
\right\|_{C^{r}}\leq 
\frac{1}{(\rho|I|)^{r}} 
\|
\boldsymbol{b} 
\|_{C^{r}},
\]
where $\|\cdot\|_{C^{r}}$ is the supremum norm of the derivatives of corresponding maps.
Next
we set 
\[
b_{u}=\boldsymbol{b} _{\frac{1}{4}, [\frac{1}{2}-\delta, \frac{1}{2}+\delta]},\ 
b_{ss}=\boldsymbol{b} _{\frac{1}{4}, [0,1]},\ 
b_{cs,k}=\boldsymbol{b} _{\frac{1}{3\tau_{cs}}, J^{\rm cs}_{k}}, 
\]
where
$\tau_{cs}=\lambda_{u}^{-1}/(1-2\lambda_{u}^{-1})$, which is independent of $k$.

For every $k\geq 1$, 
let $t_{k+1}$ be the constant given in Proposition \ref{lp_Prop_5_1}, 
the absolute value of which has an 
upper bound depending on a given $L>0$. 
We write
\[\underline{t}(L)=(t_{2},\ldots, t_{k+1},\ldots),\]
and 
define the perturbation map 
$h_{\underline{t}(L)}:\mathbb{R}^{3}\longrightarrow \mathbb{R}^{3}$
 as 
\begin{equation}\label{def-h}
h_{\underline{t}(L)}(x,y,z)=
\Bigr(x,\ y,\ z+a_{2}^{-1}b _{u}(x)\sum_{k=1}^{\infty} t_{k+1}b _{ss}(y)b _{cs,k+1}(z) \Bigr).
\end{equation}
Then we have
\begin{multline}
\left\|h_{\underline{t}(L)} -id\right\|_{C^{r}}
=
\left\|
a_{2}^{-1}b _{u}(x)\sum_{k=1}^{\infty} t_{k+1}b _{ss}(y)b _{cs,k+1}(z)
\right\|_{C^{r}}
\notag
\\
<|a_{2}|^{-1}
\left(\frac{18\tau_{ss}\tau_{cs}}{\delta}\right)^{r}
\|
\boldsymbol{b} 
\|_{C^{r}}
\sum_{k=1}^{\infty}\frac{|t_{k+1}|}{|J_{k+1}^{\rm cs}|^{r}}.
\end{multline}
Moreover, 
it follows from  \eqref{J&B} and Proposition \ref{lp_Prop_5_1} that 
\begin{multline}\label{key-r}
\sum_{k=1}^{\infty}\frac{|t_{k+1}|}{|J_{k+1}^{\rm cs}|^{r}}
\leq 
\sum_{k=1}^{\infty}
\frac{\lambda_{u}^{-(n_{0}+k+1+L(k+1))}}{
(|a_2^{-1}| \lambda_{u}^{-n_{0}-k-1})^{r}}
\\
=\frac{|a_2^{r}|\lambda_{u}^{(n_{0}+1)(r-1)}}{ \lambda_{u}^{L}}
\sum_{k=1}^{\infty}
\left(
\frac{\lambda_{u}^{r}}{\lambda_{u}^{1+L}}
\right)^{k}
=\frac{|a_2^{r}|\lambda_{u}^{(n_{0}+1)(r-1)+r}}
{\lambda_{u}^{L}(\lambda_{u}^{1+L}-\lambda_{u}^{r})}.
\end{multline}
 Consequently, 
$h_{\underline{t}(L)}$ can be taken arbitrarily $C^{r}$-close to the identity map 
if $L$ is sufficiently large as long as $r$ is fixed.

By using the perturbation map, we define 
 \begin{equation}\label{def-g}
 g=f\circ h_{\underline{t}(L)},
 \end{equation}
 which is arbitrarily $C^{1}$-close to $f$ 
 if  $L$ is large.
 
We first note that  $[\frac{1}{2}-\delta, \frac{1}{2}+\delta]\times
[0,1]\times J_{k+1}^{\rm cs}\subset U_{3\delta/2}$ for every $k\geq 1$ 
and $h_{\underline{t}(L)}|_{U_{3\delta/2}^{c}}$ is equal to the identity.
That is, $g\vert_{U_{3\delta/2}^{c}}=f$.
On the other hand, 
for any 
$(x,y,z)\in [\frac{1}{2}-\delta, \frac{1}{2}+\delta]\times
[0,1]\times J_{k+1}^{\rm cs}$, we have 
\[
g(x,y,z)=
f\circ h_{\underline{t}(L)}(x,y,z)=f(x,y,z+a_{2}^{-1} t_{k+1})\in U_{3\delta/2}^{c}.
\]
Since
$h_{\underline{t}(L)}\vert_{U_{3\delta/2}^{c}}=id$,  it follows from  \eqref{tang} that 
\[
g^{2}(x,y,z)=f^{2}(x,y,z+a_{2}^{-1} t_{k+1})=(t_{k+1}, 0, 0)+f^{2}(x,y,z).
\]
This ends the proof.
\end{proof}

\begin{rmk}
As $r=\infty$, the evaluation  \eqref{key-r} is useless. 
Hence the regularity condition in Proposition \ref{pbn} does not reach infinity.
\end{rmk}

\section{Contracting wandering domains}\label{s7}
\subsection{Two conditions in freedom term}\label{ss7.1}
From the results for $\widehat{B}^{\rm u}_{k}=B^{\rm u}(\widehat n_{k},\underline{\widehat w}^{(k)})$ obtained in Lemma \ref{l_zeta}, one can make some further conditions. 
Since  
$\underline{\widehat w}^{(k)}=\underline{w}^{(k+Lk)}\underline{v}^{(k)}\underline{\gamma}^{(k)}$, 
we have
\[
\left|\underline{w}^{(k+Lk)}\right|=n_{0}+k+Lk=O(k),\ 
\left|\underline{\gamma}^{(k)}\right|=m_{k}=O(k).
\]
Also, as in Remark \ref{rmk_freedom}, 
the sub-code $\underline{u}^{(k)}$ of 
$\underline{v}^{(k)}$ can be chosen freely.
Thus,  
we may assume that 
the length of $\underline{v}^{(k)}$ is quadratic for $k$ such that 
\begin{subequations} 
\begin{equation}\label{k2}
 \left|\underline{v}^{(k)}\right|=k^{2},
\end{equation}
which is called the \emph{quadratic condition}. 
Note that the same condition was already used in \cite{CV01}.
It follows from \eqref{k2} that  
\[
\frac{\widehat n_{k+1}}{\widehat n_{k}}
=\frac{(k+1)^{2}+O(k+1)}{k^{2}+O(k)}\to 1\ \text{as $k\to+\infty$}.
\]
That is, 
we have 
the subexponential growth in  
generations of $\widehat{B}^{\rm u}_{k}$ $(k=1,2,\ldots)$
as follows: 
\begin{lem}\label{subexp}
For any $\eta>0$, there is an integer $k_{0}>0$ such that for any integer $k\geq k_{0}$, 
\[
\widehat n_{k}<\widehat n_{k+1}<(1+\eta) \widehat n_{k}. \eqno\qed
\]
\end{lem}
\smallskip

In addition to \eqref{k2}, we have to make another condition on  
$\widehat{B}^{\rm u}_{k}=B^{\rm u}(\widehat n_{k},\underline{\widehat w}^{(k)})$.
From the freedom of the choice of $\underline{u}^{(k)}$ again, 
one can assume that the total number $\widehat n_{k(0)}$ of zeros 
in $\underline{\widehat w}^{(k)}$ is greater than or equal to 
the number $\widehat n_{k(1)}$ of ones, that is, 
\begin{equation}\label{dom}
\widehat n_{k(1)}\leq \widehat n_{k(0)},
\end{equation}
which is called the \emph{majority condition}.
\end{subequations}
\begin{rmk}
Both \eqref{k2} and \eqref{dom} are indispensable to show Lemma \ref{lem7.3} 
which is a key to  
Theorem \ref{thmWD}. 
On the other hand, 
\eqref{dom} may be an obstacle to realise some type of dynamics of wandering domain. 
See also Remark \ref{obs}.
\end{rmk}

In order to see the region in the code occupied by each symbol, 
we  sometimes denote $\widehat n_{k(0)}$
and $\widehat n_{k(1)}$ by  
$\left|\underline{\widehat w}^{(k)}\right|_{(0)}$ and $\left|\underline{\widehat w}^{(k)}\right|_{(1)}$, respectively. 
So we have  
\begin{equation}\label{dom2}
\widehat n_{k}=\widehat n_{k(0)}+\widehat n_{k(1)}=\left|\underline{\widehat w}^{(k)}\right|_{(0)}+\left|\underline{\widehat w}^{(k)}\right|_{(1)}=\left|\underline{\widehat w}^{(k)}\right|.
\end{equation}

\subsection{Identifying of wandering domains}
In the same way as in \eqref{varphi},  the following similar notations are useful here.
For any $(x,y,z)\in \mathbb{B}$ 
and integer $n>0$, we write
\[
\psi^n(x,z)=\widehat\pi(g^n(x,y,z)),\ 
\widetilde\psi^n(y)=\pi_{2}(g^n(x,y,z)) 
\]
if the value of the right-hand side of the former (resp.\ latter) equation does not depend on $y$ 
(resp.\ on $x$ or $z$),
where $\widehat\pi$ is the projection as in \eqref{varphi} and 
$\pi_{2}:\mathbb{B}\longrightarrow \mathbb{R}$ is the projection defined by $\pi_{2}(x,y,z)=y$.

To show the existence of our desired wondering domain, we have to prepare some notations. 
The first one is the following. 
For every integer $k\geq k_{0}$, we set
\[
b_{k}=a_{1}^{-1}\lambda_{u}^{-\sum_{i=0}^{\infty}\widehat n_{k+i}/2^{i}}.
\]
It implies that
\begin{equation}\label{b_k^2}
a_1\lambda_{u}^{2\widehat n_{k}}b_{k}^{2}=b_{k+1}, 
\end{equation}
which will be useful for some  evaluations later.
The next one is the following. 
Let  
$Y_{k_{0}}=[\lambda_{ss}, 1-\lambda_{ss}]$ and, 
for each integer $k> k_{0}$, 
\begin{equation}\label{Y_k}
Y_{k}=
\widetilde \psi^{\widehat n_{k-1}+2}\circ
\widetilde \psi^{\widehat n_{k-2}+2}\circ\ldots
 \circ
 \widetilde \psi^{\widehat n_{k_{0}}+2}(Y_{k_{0}}).
\end{equation}
Using these items,
for each integer $k\geq k_{0}$, 
we define
\[
\mathbb{W}_{k}=
\Bigl[\widehat x_k^{\rm u}-\frac{b_{k}}{2}, \widehat x_k^{\rm u}+\frac{b_{k}}{2}\Bigr]
\times
Y_{k}
\times
\Bigl[\frac{1}{2}-z_{k}^{\ast}, \frac{1}{2}+z_{k}^{\ast}\Bigr],
\]
where 
 $\widehat x_k^{\rm u}$ be the centre point of $\widehat B_k^{\rm u}=B^{\rm u}(\widehat n_k,\underline{\widehat w}^{(k)})$  $(k=1,2,\ldots)$ given in Proposition \ref{lp_Prop_5_1}, and 
\begin{equation} \label{z_k}
z_{k}^{\ast}=20a_{1}^{-1/2}a_{4}b_{k}^{1/2}.
\end{equation}
From the above definition of $\mathbb{W}_{k}$, 
we can immediately see that
\begin{prop}\label{proj-w}
For every $k\geq k_{0}$, 
\[
\pi_{1}\circ \widehat{\pi}(\mathbb{W}_{k})\subset 
\widehat G_k^{\rm u}=G^{\rm u}(\widehat n_k,\underline{\widehat w}^{(k)}).
\]
\end{prop}
\begin{proof}
By \eqref{ev}, \eqref{a_i} and $n_{k}<n_{k+1}$, we have
\[
(\lambda_{u}^{-1})^{\sum_{i=1}^{\infty}\widehat n_{k+i}/2^{i}}<\lambda_{u}^{-2n_{k}}
<a_{1}(1-2\lambda_{u}^{-1})\lambda_{u}^{-n_{k}},
\]
and hence 
\[
b_{k}=a_{1}^{-1}\lambda_{u}^{-\sum_{i=0}^{\infty}\widehat n_{k+i}/2^{i}}
<(1-2\lambda_{u}^{-1})(\lambda_{u}^{-1})^{\widehat n_{k}}=|\widehat G_k^{\rm u}|.
\]
Moreover, since the centre point $\widehat x_k^{\rm u}$ of $\widehat B_k^{\rm u}$ 
is identical to that of $\widehat G_k^{\rm u}$, 
the claim of this proposition has been shown.
\end{proof}
Moreover
for the diffeomorphism  
$g$ given in Proposition \ref{pbn}, 
the following result is obtained:

\begin{thm}\label{thmWD}
There is an integer $k_{1}\geq k_{0}$ such that, 
for every integer $k\geq k_{1}$, 
$\mathbb{D}_{k}=\Int (\mathbb{W}_{k})$ is a 
contracting wandering domain for $g$ 
satisfying
\[
g^{\widehat n_{k}+2}(\mathbb{D}_{k})\subset \mathbb{D}_{k+1}.
\]
\end{thm}
The proof of this theorem 
can be obtained immediately from Propositions \ref{x-width}, \ref{z-width} and \ref{y-width}.
To show them, we need two technical lemmas as follows. 
As mentioned in Remark \ref{rmk2}, 
this is the place where the partially dissipative condition \eqref{pdc} 
comes into play.
\begin{lem}\label{lem7.3}
\[
\lim_{k\to+\infty}
\left|
\frac{a_{2} \lambda_{cs0}^{\widehat n_{k(0)}}\lambda_{cs1}^{\widehat n_{k(1)}}z_{k}^{\ast}}{b_{k+1}/2}
\right|=0.
\]
%
\end{lem}
\begin{proof}
By \eqref{b_k^2} and \eqref{z_k}, we have 
\begin{align*}
\frac{a_{2} \lambda_{cs0}^{\widehat n_{k(0)}}\lambda_{cs1}^{\widehat n_{k(1)}}z_{k}^{\ast}}{2^{-1}b_{k+1}}
&=40 a_{1}^{1/2} a_{2} a_{4} \lambda_{cs0}^{\widehat n_{k(0)}}\lambda_{cs1}^{\widehat n_{k(1)}} \lambda_{u}^{-2 \widehat n_{k}} b_{k}^{-3/2}\\
&=40 a_{1}^{1/2} a_{2} a_{4} \lambda_{cs0}^{\widehat n_{k(0)}}\lambda_{cs1}^{\widehat n_{k(1)}} \lambda_{u}^{-2 \widehat n_{k}}
(a_{1} \lambda_{u}^{\sum_{i=0}^{\infty}\widehat n_{k+i}/2^{i}})^{3/2}.
\end{align*}
Let $\eta$ be any positive integer.
By Lemma \ref{subexp} 
based on the quadratic condition \eqref{k2}, there exists $k_{0}>0$ such that, for any integers $k\geq k_0$ 
and $i\geq 0$, 
$\widehat n_{k+i}<(1+\eta)^i\widehat n_k$.
Thus we have the following evaluation.
\begin{align*}
\frac{3}{2}\sum_{i=0}^{\infty}\frac{\widehat n_{k+i}}{2^{i}}
\leq 
\frac{3\widehat n_{k}}{2}\sum_{i=0}^{\infty}
\biggl(\frac{1+\eta}{2}\biggr)^{i}
=\frac{3\widehat n_{k}}{1-\eta}=(3+\eta_{1})\widehat n_{k},
\end{align*}
where $\eta_{1}=3\eta/(1-\eta)$. 
Recall that $\lambda_{cs0}\lambda_{cs1}\lambda_u^{2}<1$ by \eqref{pdc}.
One can take $\eta>0$ sufficiently small so that 
$\eta_1$ satisfies $\lambda_{cs0}\lambda_{cs1}\lambda_u^{2(1+\eta_1)}<1$.
Since $\lambda_{cs1}\lambda_u>1$ by \eqref{ev} and $\widehat n_{k(1)}\leq \widehat n_{k(0)}$ by 
\eqref{dom}, 
$(\lambda_{cs1}\lambda_u)^{\widehat n_{k(1)}}\leq (\lambda_{cs1}\lambda_u)^{\widehat n_{k(0)}}$.
It follows that
\begin{align*}
\left|
\frac{a_{2} \lambda_{cs0}^{\widehat n_{k(0)}}\lambda_{cs1}^{\widehat n_{k(1)}}z_{k}^{\ast}}{2^{-1}b_{k+1}}
\right|
&\leq 
40 a_{1}^{2} |a_{2} a_{4}| \lambda_{cs0}^{\widehat n_{k(0)}}\lambda_{cs1}^{\widehat n_{k(1)}} 
\lambda_{u}^{(1+\eta_{1})\widehat n_{k}}\\
&= 
40 a_{1}^{2} |a_{2} a_{4}|
(\lambda_{cs0}\lambda_{u}^{(1+\eta_{1})})^{\widehat n_{k(0)}}(\lambda_{cs1}\lambda_{u}^{(1+\eta_{1})})^{\widehat n_{k(1)}}, \\
&\leq 
40 a_{1}^{2} |a_{2} a_{4} |
(\lambda_{cs0}\lambda_{cs1}\lambda_{u}^{2(1+\eta_{1})})^{\widehat n_{k(0)}}\to 0\quad\text{as $k\to \infty$}.
\end{align*}
Thus the proof is now completed.
\end{proof}

We denote by $V_{\delta}$  
the $\delta$-neighbourhoods of $\{x=1/2\}\cap [0,1]^{2}$ in the $xz$-plane.
From  
 \eqref{varphi}
and
\eqref{def-g}, we have 
\begin{equation}\label{psi-a}
\left\{\begin{array}{ll}
\psi(x,z)=\varphi(x,z) &  \text{if}\ (x,z)\in [0,1]^{2}\setminus V_{\delta},\\[3pt]
\psi^{2}(x,z)=\varphi^{2}(x,z+a_{2}^{-1}t_{k+1}) & \text{if}\
 (x,z)\in V_{\delta}.
\end{array}\right.
\end{equation}
Let $\widehat x_k^{\rm u}$ be the centre point of $\widehat B_k^{\rm u}=B^{\rm u}(\widehat n_k,\underline{\widehat w}^{(k)})$  $(k=1,2,\ldots)$ given in Proposition \ref{lp_Prop_5_1}.

\begin{lem}\label{lem7.4}
For any $(\widehat x_k^{\rm u}+x, 1/2+z)\in \widehat B_k^{\rm u}\times[0,1]$,
\begin{multline}
\psi^{\widehat n_k+2}(\widehat x_k^{\rm u}+x, 1/2+z)
=(\widehat x_{k+1}^{\rm u}, 1/2)+
\Bigl(
-a_{1}\lambda_{u}^{2\widehat n_k} x^{2}
+a_{2}\lambda_{cs0}^{\widehat n_{k(0)}} \lambda_{cs1}^{\widehat n_{k(1)}}z, \notag\\
a_{4}(-1)^{\widehat n_{k(1)}}\lambda_{u}^{\widehat n_{k}}x
\Bigr).
\end{multline}
\end{lem}
\begin{proof}
For simplicity, let us here write 
$\underline{\widehat w}^{(k)}=w_{1}w_{2}\ldots w_{i}\ldots w_{\widehat n_k}$.
Let $\xi_0$, $\xi_1$ be the functions on $\mathbb{R}$ defined by $\xi_0(x)=\lambda_ux$ and 
$\xi_1(x)=\lambda_u(1-x)$.
Then, for any $\widehat x,x\in \mathbb{R}$, 
\[
\xi_0(\widehat x+x)=\lambda_u(\widehat x+x)
=\xi_0(\widehat x)+\lambda_ux,\ 
\xi_1(\widehat x+x)=\lambda_u(1-\widehat x-x)
=\xi_1(\widehat x)-\lambda_ux.
\]
Similarly, by \eqref{IFS}, for any $\alpha,z$ such that 
$\alpha+z$ and $z$ are in the domains of the corresponding functions,
$$\zeta_0(\alpha+z)=\zeta_0(\alpha)+\lambda_{cs0}z,\quad 
\zeta_1(\alpha+z)=\zeta_1(\alpha)+\lambda_{cs1}z.$$
Hence, by the first equation of \eqref{psi-a} together with 
\eqref{blender} and \eqref{varphi}, for each $i\in\{1,\ldots,\widehat n_{k}\}$, 
\begin{multline}\label{eq7a}
\psi^{i}(\widehat x_k^{\rm u}+x, 1/2+z)
=
\Bigl(
\xi_{w_i}\circ\ldots\circ\xi_{w_2}\circ\xi_{w_1}(\widehat x_k^{\rm u})+
\lambda_{u}^{\widehat n_{i(0)}}
(-\lambda_{u})^{\widehat n_{i(1)}}
x, \\
\zeta_{w_i}\circ\ldots\circ\zeta_{w_2}\circ\zeta_{w_1}(1/2)
+
\lambda_{cs0}^{\widehat n_{i(0)}}
\lambda_{cs1}^{\widehat n_{i(1)}}
z
\Bigr).
\end{multline}
Since $\widehat x_k^{\rm u}$ is the centre point of $\widehat B_k^{\rm u}$, 
\[
\xi_{w_{\widehat n_{k}}}\circ\ldots\circ\xi_{w_2}\circ\xi_{w_1}(\widehat x_k^{\rm u})=1/2.
\]
Moreover,  by \eqref{t_{k+1}},  
\[\zeta_{w_{\widehat n_k}}\circ\ldots\circ\zeta_{w_2}\circ\zeta_{w_1}(1/2)=
\boldsymbol{\zeta}_{\underline{\widehat w}^{(k)}}(1/2)=\hat z_{k+1}^{cs}-a_2^{-1}t_{k+1}.
\]
Since $\widehat n_{k}=\left|\underline{\widehat w}^{(k)}\right|_{(0)}+\left|\underline{\widehat w}^{(k)}\right|_{(1)}=
	\widehat n_{k(0)}+\widehat n_{k(1)}$ by \eqref{dom2}, 
the equation \eqref{eq7a} shows that 
\[
\psi^{\widehat n_{k}}(\widehat x_k^{\rm u}+x, 1/2+z)
=\Bigl(
1/2+\lambda_{u}^{\widehat n_{k(0)}}(-\lambda_{u})^{\widehat n_{k(1)}}x, \\
\hat z_{k+1}^{cs}-a_2^{-1}t_{k+1}+
\lambda_{cs0}^{\widehat n_{k(0)}}\lambda_{cs1}^{\widehat n_{k(1)}}z
\Bigr)\in V_{\delta}.
\]
By the second equation of \eqref{psi-a}, 
\begin{align*}
\psi^{2}\circ \psi^{\widehat n_{k}}&(\widehat x_k^{\rm u}+x, 1/2+z)
=
\varphi^{2} \left(\psi^{\widehat n_{k}}(\widehat x_k^{\rm u}+x, 1/2+z)+(0,a_{2}^{-1}t_{k+1})
\right) \\
& =\varphi^{2}
\left(
1/2+\lambda_{u}^{\widehat n_{k(0)}}(-\lambda_{u})^{\widehat n_{k(1)}}x,\ 
\hat z_{k+1}^{cs}-a_2^{-1}t_{k+1}+
\lambda_{cs0}^{\widehat n_{k(0)}}\lambda_{cs1}^{\widehat n_{k(1)}}z+a_{2}^{-1}t_{k+1}
\right)\\ 
& =\varphi^{2}
\left(
1/2+(-1)^{\widehat n_{k(1)}}\lambda_{u}^{\widehat n_{k}}x,\ 
\widehat z_{k+1}^{\rm cs}+
\lambda_{cs0}^{\widehat n_{k(0)}}\lambda_{cs1}^{\widehat n_{k(1)}}z
\right) , \\
\intertext{by \eqref{tang}, \eqref{varphi} 
 and \eqref{dom2},}
&=
\left(
-a_{1}\lambda_{u}^{2\widehat n_{k}}x^{2}
+a_{2}\lambda_{cs0}^{\widehat n_{k(0)}}\lambda_{cs1}^{\widehat n_{k(1)}}z+a_{2}\widehat z_{k+1}^{\rm cs},\  
 a_{4}(-1)^{\widehat n_{k(1)}}\lambda_{u}^{\widehat n_{k}}x+1/2
\right). 
\end{align*}
Since $a_{2}\widehat z_{k+1}^{\rm cs}=\widehat x_{k+1}^{\rm u}$ from \eqref{J&B}, we have obtained 
the equation required in this lemma.
\end{proof}

For each $k>0$, we have the rectangle 
$W_{k}=\widehat\pi(\mathbb{W}_{k})$ with    
the sides
$\partial_{z}W_{k}=
\widehat\pi
(
\mathbb{W}_{k}\cap \{z=1/2\pm z_{k}^{\ast}\}
)$,  
$\partial_{x}W_{k}=
\widehat\pi
(
\mathbb{W}_{k}\cap \{x=\widehat x_k^{\rm u}\pm b_{k}/2 \}
)
$ 
and the central line 
$\boldsymbol{c}(W_{k})=\widehat\pi(\mathbb{W}_{k}\cap \{x=\widehat x_k^{\rm u}\})$.
See Figure \ref{xz}.
\begin{figure}[hbtp]
\centering
\scalebox{0.9}{\includegraphics[clip]{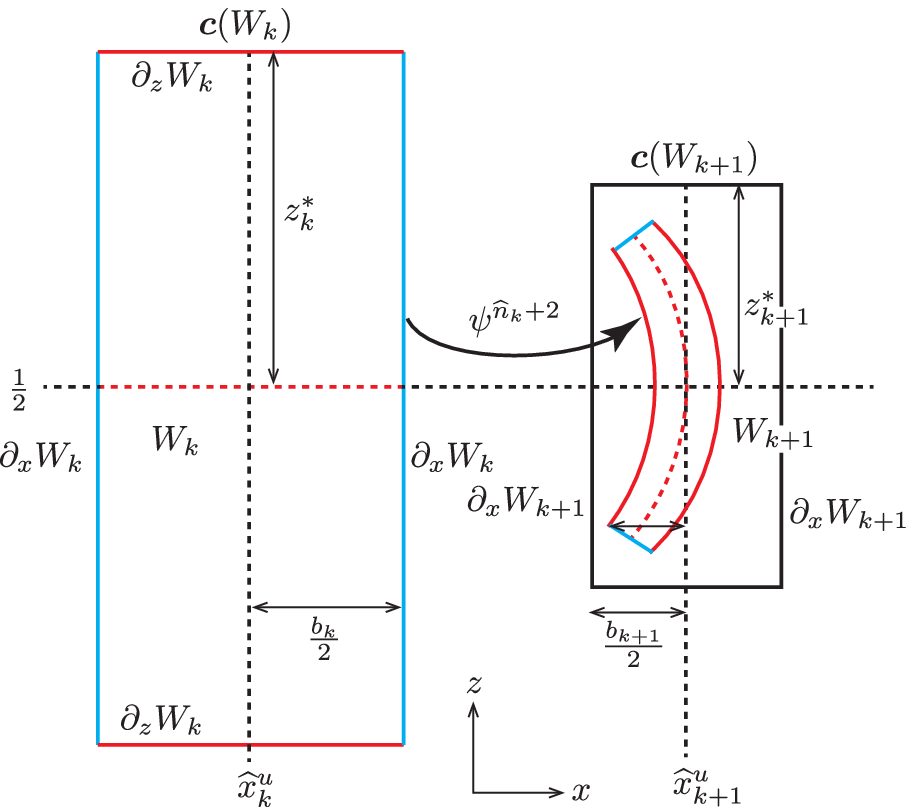}}
\caption{}
\label{xz}
\end{figure}

\begin{prop}\label{x-width}
There is an integer $k_{0}^{'}\geq k_{0}$ such that, for any 
integer $k>k_{0}^{'}$, 
\[
\pi_{1}(\psi^{\widehat n_k+2}(W_{k}))\subset \pi_{1}(W_{k+1}),
\]
where $\pi_{1}$ is the projection with 
$\pi_{1}(x,z)=x$.
\end{prop}
\begin{proof}
From the form \eqref{tang}, $\psi^{\widehat n_k+2}(\partial_{z}W_{k})$ consists of 
two quadratic curves. See Figure \ref{xz}.
Points of $\psi^{\widehat n_k+2}(W_{k})$ furthest from $\boldsymbol{c}(W_{k+1})$ 
are endpoints of one of the quadratic curves.
By Lemma \ref{lem7.4} and \eqref{b_k^2}, 
we have
\begin{align*}
d_{h}\left(\boldsymbol{c}(W_{k+1}), \psi^{\widehat n_k+2}(W_{k})\right)
&=
a_{1}(\lambda_{u}^{\widehat n_k}b_{k}/2)^{2}+\left|a_{2}\lambda_{cs0}^{\widehat n_{k(0)}}\lambda_{cs1}^{\widehat n_{k(1)}}z_{k}^{\ast}\right|\\
&
=4^{-1}b_{k+1}+\left|a_{2}\lambda_{cs0}^{\widehat n_{k(0)}}\lambda_{cs1}^{\widehat n_{k(1)}}z_{k}^{\ast}\right|,
\end{align*}
where $d_{h}$ is the Hausdorff distance of the two subsets. 
It follows from \eqref{b_k^2} and Lemma \ref{lem7.3} that the width comparison along the $x$-direction is the following:
\[
\frac{d_{h}\left(\boldsymbol{c}(W_{k+1}), \psi^{\widehat n_k+2}(W_{k})\right)}{d_{h}\left(\boldsymbol{c}(W_{k+1}), \partial_{x}W_{k+1}\right)}
= 
\frac{1}{2}+
\left|
\frac{a_{2}
\lambda_{cs0}^{\widehat n_{k(0)}}\lambda_{cs1}^{\widehat n_{k(1)}}z_{k}^{\ast}}{b_{k+1}/2}
\right|.
\]
Note that, from Lemma \ref{lem7.3}, the right-hand side of the inequality is less than $1$ if one takes $k$ sufficiently large.
This proves the desired assertion and completes the proof of the proposition.
\end{proof}

\begin{prop}\label{z-width}
There is an integer $k_{0}^{''}\geq k_{0}$ such that, for any 
integer $k>k_{0}^{''}$, 
\[
\pi_{3}(\psi^{\widehat n_k+2}(W_{k}))\subset \pi_{3}(W_{k+1}),
\]
where $\pi_{3}$ is the projection with 
$\pi_{3}(x,z)=z$.
\end{prop}

\begin{proof}
By the same reason stated in the beginning of the proof of Proposition \ref{x-width}, 
it is sufficient to evaluate how the endpoints of components of 
$\psi^{\widehat n_k+2}(\partial_{z}W_{k})$  
are far from $\{z=1/2\}$. More concretely, 
it follows from  Lemma \ref{lem7.4} that 
it suffices to prove that the following inequality:
\[
\left|
a_{4}\lambda_{u}^{n_{k}}b_{k}/2
\right|<z_{k+1}^{\ast}/2.
\]
By \eqref{z_k}, it is equivalent to 
\[
a_{1}\lambda_{u}^{2n_{k}}b_{k}^{2}<400 b_{k+1}.
\]
This is established from \eqref{b_k^2}  and  the proof is accomplished.
\end{proof}

Now let us turn our attention to $Y_{k}$ defined in 
\eqref{Y_k}.
\begin{prop}\label{y-width}
For every integer $k>k_{0}$, 
$Y_{k}$ is contained in $\bigl(\frac{1}{2}-\frac{a_{3}}{2}, \frac{1}{2}+\frac{a_{3}}{2}\bigr)$ and 
\[
\lim_{k\to+\infty} 
\left|Y_{k}
\right|=0.
\]
\end{prop}
\begin{proof}

For the generation 
$\widehat n_{k_{0}}$  of  
$\widehat B_{k_{0}}^{\rm u}=B^{\rm u}(\widehat n_{k_{0}},\underline{\widehat w}^{(k_{0})})$, we have  
\[
\left|
\widetilde\psi^{\widehat n_{k_{0}}}(Y_{k_{0}})
\right|
=\lambda_{ss}^{\widehat n_{k_{0}}}\left|Y_{k_{0}}\right|
=\lambda_{ss}^{\widehat n_{k_{0}}}(1-2\lambda_{ss}).
\]
From \eqref{tang} together with \eqref{a_i}, 
$
Y_{k_{0}+1}=
\widetilde\psi^{\widehat n_{k_{0}}+2}(Y_{k_{0}})
\subset \bigl(\frac{1}{2}-\frac{a_{3}}{2}, \frac{1}{2}+\frac{a_{3}}{2}\bigr)
$
and
\[
\left|
Y_{k_{0}+1}
\right|
=|a_{3}|\lambda_{ss}^{\widehat n_{k_{0}}}|Y_{k_{0}}|
=|a_{3}|\lambda_{ss}^{\widehat n_{k_{0}}}(1-2\lambda_{ss}).
\]
By inductive steps, one can show that, for every integer $k>k_{0}$, 
$Y_{k}\subset \bigl(\frac{1}{2}-\frac{a_{3}}{2}, \frac{1}{2}+\frac{a_{3}}{2}\bigr)$ and 
\[
\left|
Y_{k}
\right|
=|a_{3}|^{k-k_{0}}
\lambda_{ss}^{\sum_{i=0}^{k-k_{0}} \widehat n_{k+i}}
(1-2\lambda_{ss}).
\]
Hence, it converges to $0$ as $k\to+\infty$. 
\end{proof}
\begin{proof}[Proof of Theorem \ref{thmWD}]
From Propositions \ref{x-width} and \ref{z-width}, 
there is  an integer $k_{1}\geq k_{0}$ such that, 
for any integer $k\geq k_{1}$,
\[
\psi^{\widehat n_k+2}(W_{k})\subset \Int (W_{k+1}),
\]
and moreover 
\[
\lim_{k\to+\infty}\mathrm{diam}(W_{k+1})=0.
\]
On the other hand, Proposition \ref{y-width} implies that, 
for any $k>k_{1}$, diameter of $Y_{k}$  converges to zero as $k\to+\infty$. 
Since $W_{k}\times Y_{k}$  is equal to $\mathbb{W}_{k}$, the proof is complete.
\end{proof}

\section{Probabilistic representations}\label{s8}
Let $g$ be the diffeomorphism defined in \eqref{def-g},  
$\mathbb{D}_{k}=\Int (\mathbb{W}_{k})$ the contracting wandering domain of $g$ 
and $k_{1}$ the integer obtained  in Theorem \ref{thmWD}. 
Also with Proposition \ref{proj-w}, 
for each $k> k_{1}$, 
\begin{subequations}
\begin{equation}\label{eq8-1}
g^{\widehat n_{k}+2}(\mathbb{D}_{k})\subset \mathbb{D}_{k+1},\quad  
\pi_{1}\circ \widehat{\pi}(\mathbb{D}_{k})\subset \widehat 
B_k^{\rm u}=B^{\rm u}(\widehat n_{k},\underline{\widehat w}^{(k)}),
\end{equation}
where $\widehat{\pi}$   and $\pi_{1}$ 
are the projections given as in \eqref{varphi} and 
Proposition \ref{x-width}, respectively, and  $\widehat n_{k}=\left|\underline{\widehat w}^{(k)}\right|$. 
The itinerary $\underline{\widehat w}^{(k)}$ 
consists of three parts as
\begin{equation}\label{eq8-2}
\underline{\widehat w}^{(k)}=\underline{w}^{(k+Lk)}\underline{v}^{(k)}\underline{\gamma}^{(k)},
\end{equation} 
where 
the sub-code 
$\underline{v}^{(k)}$ is a 
$k^{2}$-tuple  $\underline{v}^{(k)}=(v_{1}v_{2}\ldots v_{k^{2}})$,  at least $k^{2}-1$ elements of which
can be chosen freely, 
and the other parts satisfy 
\begin{equation}\label{eq8-3}
\left|\underline{w}^{(k+Lk)}\right|=n_{0}+k+Lk,\ 
\left|\underline{\gamma}^{(k)}\right|=m_{k}, 
\end{equation}
\end{subequations}
where $n_{0}+k+Lk$ and $m_{k}$ are integers given in 
Lemma \ref{l_zeta}. 
Let us now take advantage of this freedom of  $\underline{v}^{(k)}$ 
to realise historicity and 
physicality.

\subsection{Historicity}
The following theorem guarantees half of the claim in Theorem \ref{thm1}, the part about historicity.
\begin{thm}\label{hist}
There exists a sequence $\boldsymbol{v}=(\underline{v}^{(k)})_{k> k_{1}}$ of codes
such that 
$\mathbb{D}_{k}$ 
is a historic contracting wandering domain  for $g=g_{\boldsymbol{v}}$.
\end{thm}
To show this claim we needs 
the following two conditions:
\begin{description}
\item[Era condition] 
We consider an increasing sequence $(k_{s})_{s\in \mathbb{N}}$ of integers,
which satisfies the following condition: 
for every $s\in \mathbb{N}$,
\begin{equation}\label{era-ratio}
\sum_{k=k_{s}}^{k_{s+1}-1} k^{2}>s \sum_{k=k_{1}}^{k_{s}-1} k^{2}.
\end{equation}
Note that this setting provides us with a situation 
that the new era from 
$k_{s}$ until $k_{s+1}-1$ to be so dominant as to ignore the old one from $k_{1}$ until $k_{s}-1$,
see Claim \ref{clam1}.
\item[Code condition (for historic behaviour)]
On the era condition,
for any integer $k=k(s)$ with $k_{s}<k\leq k_{s+1}$, 
we consider each entry of $\underline{v}^{(k)}=(v_{1}v_{2}\ldots v_{k^{2}})$ 
 satisfying \eqref{dom}  and the following rules: 
 \begin{subequations}
\begin{itemize}
\item 
if $s$ is even, 
\begin{equation}\label{h1}
v_{i}=\left\{
\begin{array}{ll}
0 & \text{for}\ i=1,\ldots, \left\lfloor 3k(s)^{2}/4\right\rfloor
\\[3pt]
1 & \text{for}\ i=\lceil 3k(s)^{2}/4\rceil,\ldots, k(s)^{2},
\end{array}\right.
\end{equation} 
that is, 
$\underline{v}^{(k)}=
\overbrace{000\ldots\ldots0}^{\left\lfloor 3k^{2}/4\right\rfloor}
\overbrace{1\ldots 1}^{\left\lceil k^{2}/4\right\rceil}$, 
\item
if $s$ is odd, 
\begin{equation}\label{h2}
v_{i}=\left\{
\begin{array}{ll}
0 & \text{for}\ i=1,\ldots, \left\lfloor 7k(s)^{2}/8\right\rfloor
\\[3pt]
1 & \text{for}\  i=\lceil 7k(s)^{2}/8\rceil,\ldots, k(s)^{2},
\end{array}\right.
\end{equation} 
that is, $\underline{v}^{(k)}=
\overbrace{000\ldots\ldots0}^{\left\lfloor 7k^{2}/8\right\rfloor}
\overbrace{1\ldots 1}^{\left\lceil k^{2}/8\right\rceil}$,  
\end{itemize}
\end{subequations}
where $\left\lfloor\cdot \right\rfloor$ and $\lceil\cdot\rceil$ indicate the 
floor and ceiling functions, respectively. 
\end{description}
Note that both \eqref{h1} and \eqref{h2} satisfy 
the quadratic condition \eqref{k2} and the majority condition \eqref{dom}.

It may be obvious that historic behaviour appears under \eqref{h1} and \eqref{h2}. 
Indeed, in \cite{CV01, KS17} they gave constructions similar to ours, 
but did not provide detailed proofs. However, 
we here describe a proof in detail for the convenience of readers.

\begin{proof}[Proof of Theorem \ref{hist}]
The proof is carried out under the era and code conditions, 
which are required to show Claims \ref{clam1} and \ref{clam2}.

For given non-negative integers $n$, $m$ with $n<m$ and 
$x\in \mathbb{D}_{k_{1}}$, the  \emph{empirical probability measure} 
is defined by
\[\nu_{x}^{(n,m)}=\frac{1}{m-n}\sum_{i=n}^{m-1}\delta_{g^{i}(x)}.\]
For any  integer $k_{s}>0$, we write 
\[
\widehat N_{k_{s}}=\sum_{k=k_{1}}^{k_{s}-1} (\widehat n_{k}+2).
\]
Let  $\widehat{\mathbb{B}}$ be a compact subset of $\mathbb{R}^{3}$ containing $\bigcup_{i=0}^{2}g^{i}(\mathbb{B})$. 
For any $\varPhi \in C^{0}(\widehat{\mathbb{B}},\mathbb{R})$,  we have
\begin{equation}\label{eq7.1}
\int \varPhi d\nu_{x}^{(\widehat N_{k_{s}}, \widehat N_{k_{s+1}})}
=\frac{1}{\widehat N_{k_{s+1}}-\widehat N_{k_{s}}}\sum_{i=\widehat N_{k_{s}}}^{\widehat N_{k_{s+1}}-1}\varPhi\circ g^{i}(x).
\end{equation}
\begin{clam}\label{clam1}
For any  $x\in \mathbb{D}_{k_{1}}$, 
\[
\lim_{s\to+\infty}\left|
\int \varPhi d\nu_{x}^{(\widehat N_{k_{s}},\widehat N_{k_{s+1}})}
-
\int \varPhi d\nu_{x}^{(0,\widehat N_{k_{s+1}})}
\right|=0.
\]
\end{clam}
\noindent
Here we show the claim.
Consider 
\[
\left|A_{s}+B_{s}\right|=
\left|
\int \varPhi d\nu_{x}^{(\widehat N_{k_{s}},\widehat N_{k_{s+1}})}
-
\int \varPhi d\nu_{x}^{(0,\widehat N_{k_{s+1}})}
\right|,
\]
where 
\begin{align*}
& A_{s}=\frac{1}{\widehat N_{k_{s+1}}-\widehat N_{k_{s}}}\sum_{i=\widehat N_{k_{s}}}^{\widehat N_{k_{s+1}}-1}\varPhi\circ g^{i}(x)
-
\frac{1}{\widehat N_{k_{s+1}}}\sum_{i=\widehat N_{k_{s}}}^{\widehat N_{k_{s+1}}-1}\varPhi\circ g^{i}(x),
\\
& B_{s}=
\frac{1}{\widehat N_{k_{s+1}}}\sum_{i=\widehat N_{k_{s}}}^{\widehat N_{k_{s+1}}-1}\varPhi\circ g^{i}(x)
-
\frac{1}{\widehat N_{k_{s+1}}}\sum_{i=0}^{\widehat N_{k_{s+1}}-1}\varPhi\circ g^{i}(x).
\end{align*}
Thus, 
the proof will be complete if 
 $|A_{s}|$ and $|B_{s}|$ converge to $0$ as $s\to 0$. 
In fact, 
\begin{multline*}
|A_{s}|\leq
\left|
\frac{(\widehat N_{k_{s+1}}-\widehat N_{k_{s}})
(\widehat N_{k_{s+1}}-(\widehat N_{k_{s+1}}-\widehat N_{k_{s}}))
\|\varPhi\|_{C^{0}}}{
(\widehat N_{k_{s+1}}-\widehat N_{k_{s}})\widehat N_{k_{s+1}}
}
\right|
\\
=
\frac{\widehat N_{k_{s}}}{\widehat N_{k_{s+1}}}\|\varPhi\|_{C^{0}}
<
\frac{1}{1+s} \|\varPhi\|_{C^{0}},
\end{multline*}
where the last inequality follows from \eqref{era-ratio}. On the other hand,
\[
|B_{s}|
\leq
\frac{\left|(\widehat N_{k_{s+1}}-\widehat N_{k_{s}})-\widehat N_{k_{s+1}}\right|\|\varPhi\|_{C^{0}}
}{\widehat N_{k_{s+1}}}
=\frac{\widehat N_{k_{s}}}{\widehat N_{k_{s+1}}}\|\varPhi\|_{C^{0}}
<\frac{1}{1+s} \|\varPhi\|_{C^{0}}.
\]
Hence, $|A_{s}|, |B_{s}|\to 0$ as $s\to +\infty$.
This ends the proof of Claim \ref{clam1}.
\medskip

Based on the result of Claim \ref{clam1},  we focus  only on \eqref{eq7.1}.
So we divide it into three parts as follows:
\[
\frac{1}{\widehat N_{k_{s+1}}-\widehat N_{k_{s}}}
\sum_{i=\widehat N_{k_{s}}}^{\widehat N_{k_{s+1}}-1} \varPhi(g^{i}(x))=
\frac{1}{\widehat N_{k_{s+1}}-\widehat N_{k_{s}}}
(S_{1}+S_{2}+S_{3}),\] 
where 
\begin{align*}
S_{1}&=\sum_{i=\widehat N_{k_{s}}}^{\widehat N_{k_{s}}+(n_{0}+k_{s}+Lk_{s})-1}\varPhi(g^{i}(x)),\quad
S_{2}=\sum_{i=\widehat N_{k_{s}}+(n_{0}+k_{s}+Lk_{s})}^{\widehat N_{k_{s}}+(n_{0}+k_{s}+Lk_{s})+k_{s}^{2}-1}\varPhi(g^{i}(x)),\\
S_{3}&=\sum_{i=\widehat N_{k_{s}}+(n_{0}+k_{s}+Lk_{s})+k_{s}^{2}}^{\widehat N_{k_{s+1}}-1}\varPhi(g^{i}(x)).
\end{align*}
Note that 
the number of terms in the sum of $S_{1}$ and $S_{3}$ 
is $O(k_{s})$, while that of $S_{2}$ is $k_{s}^{2}$. 
Since  
\[\widehat N_{k_{s+1}}-\widehat N_{k_{s}}=\widehat n_{k_{s}}+2=(n_{0}+k_{s}+Lk_{s})+k_{s}^{2}+m_{k_{s}}+2=k_{s}^{2}+O(k_{s}),\]
we have  
\begin{multline}\label{eq7.2}
\lim_{s\to +\infty}
\left|
\frac{1}{\widehat N_{k_{s+1}}-\widehat N_{k_{s}}}
\sum_{i=\widehat N_{k_{s}}}^{\widehat N_{k_{s+1}}} \varPhi(g^{i}(x))
-\frac{S_{2}}{k_{s}^{2}}\right|
\\
=\lim_{s\to +\infty}
\left|\left(
\frac{S_{1}}{\widehat n_{k_{s}}+2}
+\frac{k_{s}^{2}}{\widehat n_{k_{s}}+2}\left(\frac{S_{2}}{k_{s}^{2}}\right)
+\frac{S_{3}}{\widehat n_{k_{s}}+2}
\right)-\frac{S_{2}}{k_{s}^{2}}
\right|
=0.
\end{multline}
For simplicity, write
$
\widehat x:=g^{\widehat N_{k}+(n_{0}+k_{s}+Lk_{s})}(x)
$
and hence  
\[S_{2}=\sum_{j=0}^{k_{s}^{2}-1}\varPhi(g^{j}(\widehat x)).\]

It is sufficient to prove the following claim for Theorem \ref{hist}:  
\begin{clam}\label{clam2}
\[
\lim_{s\to+\infty} \frac{S_{2}}{k_{s}^{2}}
=\left\{
\begin{array}{ll}
(3\varPhi(P_{g})+\varPhi(Q_{g}))/4  & \text{if $s$ is even},\\[3pt]
(7\varPhi(P_{g})+\varPhi(Q_{g}))/8 & \text{if $s$ is odd},
\end{array}\right.
\]
where $P_{g}$ and $Q_{g}$ are the continuations of the fixed points $P$ and $Q$, respectively.
\end{clam}
\noindent
To prove this claim, 
define 
\begin{align*}
N_{s}^{(P_{g})}=N_{s}^{(P_{g})}(\rho)
&=\max\left\{N>0:\ 
\text{$g^{i}(\widehat x)\in U_{\rho}(P_{g})$ for $0\leq \forall i\leq N$}
\right\},\\
\widetilde{N}_{s}^{(P_{g})}
&=\max \left\{N>0:\ 
\text{$g^{i}(\widehat x)\in \mathbb{V}_{0}$ for $0\leq \forall i\leq N$}
\right\},\\
N_{s}^{(Q_{g})}=N_{s}^{(Q_{g})}(\rho)
&=\max\left\{N>0:\ 
\text{$g^{i}(\widehat x)\in U_{\rho}(Q_{g})$ for $N_{s}^{(P_{g})}\leq \forall i< N$}\right\},
\end{align*}
where $U_{\rho}(P_{g})$ and 
$U_{\rho}(Q_{g})$ are the $\rho$-neighbourhoods of $P_{g}$ and 
$Q_{g}$, respectively, 
for a given constant $\rho>0$, 
and $\mathbb{V}_{0}$ is the component of $g^{-1}(\mathbb{B})\cap \mathbb{B}$ containing $P_{g}$. 
Using them, we have
\begin{multline*}
\sum_{i=0}^{k_{s}^{2}-1}\varPhi(g^{i}(\widehat x))
= 
\sum_{i=0}^{N_{s}^{(P_{g})}-1} \varPhi(g^{i}(\widehat x))
+\sum_{i=N_{s}^{(P_{g})}}^{\widetilde N_{s}^{(P_{g})}-1} \varPhi(g^{i}(\widehat x))\\
+\sum_{i=\widetilde N_{s}^{(P_{g})}}^{N_{s}^{(Q_{g})}-1} \varPhi(g^{i}(\widehat x))
+\sum_{i=N_{s}^{(Q_{g})}}^{k_{s}^{2}-1}\varPhi(g^{i}(\widehat x)).
\end{multline*}
For any small $\varepsilon>0$, there is a $\rho>0$ such that 
\[
\left|
\frac{\sum_{j=0}^{N_{s}^{(P_{g})}-1}\varPhi(g^{j}(\widehat x))}{N_{s}^{(P_{g})}}-\varPhi(P_{g})
\right|
<\varepsilon,\quad 
\left|
\frac{
\sum_{i=\widetilde N_{s}^{(P_{g})}}^{N_{s}^{(Q_{g})}-1}\varPhi(g^{j}(\widehat x))
}{
N_{s}^{(Q_{g})}-\widetilde N_{s}^{(P_{g})}
}-\varPhi(Q_{g})
\right|
< \varepsilon.
\]
It implies that 
\begin{multline*}
\frac{1}{k_{s}^{2}}\sum_{i=0}^{k_{s}^{2}-1}\varPhi(g^{i}(\widehat x))
<
\frac{ N_{s}^{(P_{g})}}{\widetilde N_{s}^{(P_{g})}}
\frac{\widetilde N_{s}^{(P_{g})}}{k_{s}^{2}}(\varPhi(P_{g})+\varepsilon)
+\frac{1}{k_{s}^{2}}
\sum_{i=N_{s}^{(P_{g})}}^{\widetilde N_{s}^{(P_{g})}-1} \varPhi(g^{i}(x))
\\
+
\frac{1}{k_{s}^{2}}
\sum_{i=N_{s}^{(Q_{g})}}^{k_{s}^{2}-1} \varPhi(g^{i}(x))
+\frac{N_{s}^{(Q_{g})}-\widetilde N_{s}^{(P_{g})}}{k_{s}^{2}-\widetilde N_{s}^{(P_{g})}}
\frac{k_{s}^{2}-\widetilde N_{s}^{(P_{g})}}{k_{s}^{2}}(\varPhi(Q_{g})+\varepsilon),
\end{multline*}
and 
\begin{multline*}
\frac{1}{k_{s}^{2}}\sum_{i=0}^{k_{s}^{2}-1}\varPhi(g^{i}(\widehat x))
>
\frac{ N_{s}^{(P_{g})}}{\widetilde N_{s}^{(P_{g})}}
\frac{\widetilde N_{s}^{(P_{g})}}{k_{s}^{2}}(\varPhi(P_{g})-\varepsilon)
+
\frac{1}{k_{s}^{2}}
\sum_{i=N_{s}^{(P_{g})}}^{\widetilde N_{s}^{(P_{g})}-1}\varPhi(g^{i}(x))\\
+
\frac{1}{k_{s}^{2}}
\sum_{i=N_{s}^{(Q_{g})}}^{k_{s}^{2}-1}\varPhi(g^{i}(x))
+\frac{N_{s}^{(Q_{g})}-\widetilde N_{s}^{(P_{g})}}{k_{s}^{2}-\widetilde N_{s}^{(P_{g})}}
\frac{k_{s}^{2}-\widetilde N_{s}^{(P_{g})}}{k_{s}^{2}}(\varPhi(Q_{g})-\varepsilon).
\end{multline*}

Since $\rho$ is already fixed, 
it follows from code conditions \eqref{h1} and \eqref{h2} that 
 \[ 
 \lim_{s\to +\infty} \frac{N_{s}^{(P_{g})}}{\widetilde{N}_{s}^{(P_{g})}}=
 \lim_{s\to +\infty} \frac{N_{s}^{(Q_{g})}-\widetilde N_{s}^{(P_{g})}}{k_{s}^{2}-\widetilde N_{s}^{(P_{g})}}
 =1,
 \]
and 
\[
 \lim_{s\to +\infty}\frac{\widetilde N_{s}^{(P_{g})}}{k_{s}^{2}}=
  \left\{\begin{array}{ll}
  3/4 & \text{if $s$ is even,} \\[2pt]
 7/8 & \text{if $s$ is odd,}\end{array}\right.\quad
 \lim_{s\to +\infty}\frac{k_{s}^{2}-\widetilde N_{s}^{(P_{g})}}{k_{s}^{2}}= \left\{\begin{array}{ll}
  1/4 & \text{if $s$ is even,} \\[2pt]
 1/8 & \text{if $s$ is odd.}\end{array}\right.
\]
Moreover, 
\[
\lim_{s\to +\infty}
\frac{1}{k_{s}^{2}}
\left(
\sum_{i=N_{s}^{(P_{g})}}^{\widetilde N_{s}^{(P_{g})}-1}
\varPhi(g^{i}(x))
+
\sum_{i=N_{s}^{(Q_{g})}}^{k_{s}^{2}-1}
\varPhi(g^{i}(x))
\right)
= 0.
\]
This finishes the proof of Claim \ref{clam2}.
\medskip

Finally, 
by Claims \ref{clam1} and \ref{clam2} together with \eqref{eq7.2},
\[
\lim_{s\to+\infty}
\int \varPhi d\nu_{x}^{(0,\widehat N_{k_{s+1}})}
=\left\{
\begin{array}{ll}
(3\varPhi(P_{g})+\varPhi(Q_{g}))/4  & \text{if $s$ is even},\\[3pt]
(7\varPhi(P_{g})+\varPhi(Q_{g}))/8 & \text{if $s$ is odd}.
\end{array}\right.
\]
We complete the proof of Theorem \ref{hist}.
\end{proof}

\subsection{Physicality}
The final discussion in this 
paper concerns the existence of non-trivial Dirac physical measure 
associated with a contracting  wandering domain in Theorem \ref{thm1}. 
To show it we need not to take any era condition as \eqref{era-ratio} into account. 
Moreover instead of \eqref{h1} and \eqref{h2}  
we adopt simpler code condition, which is the same as that  in  \cite[Section 9]{CV01}, as follows:
\begin{description}
\item[Code condition (for Dirac physical measure supported on $P_{g}$)] For a given integer $k>0$, 
we suppose that 
the freedom part $\underline{v}^{(k)}$ of
the itinerary $\underline{\widehat w}^{(k)}$ in \eqref{eq8-2}
consists of  
$k^{2}$ zeros, that is, 
\begin{equation}\label{code3}
\underline{v}^{(k)}=\underline{0}^{k^{2}}=\overbrace{
0000\ldots 00}^{k^{2}}.
\end{equation}
\end{description}
Note that  \eqref{code3} is not contradict to 
the quadratic condition \eqref{k2} and the majority condition \eqref{dom}.

\begin{rmk}\label{obs}
On the other hand, 
since the itinerary $\underline{1}^{k^{2}}$ does not meet  \eqref{dom}, 
 the other saddle fixed point
$Q_{g}$ may not be a support of the Dirac physical measure. See also  Remark \ref{tg}.
\end{rmk}

\begin{thm}\label{phys}
There exist an integer $k_{2}>0$ and a 
sequence $\boldsymbol{v}=(\underline{v}^{(k)})_{k> \max\{k_{1}, k_{2}\}}$ of codes 
such that $g=g_{\boldsymbol{v}}$  has the non-trivial Dirac physical measure supported on $P_{g}$ associated with 
the contracting wandering domain $\mathbb{D}_{k}$.
\end{thm}
\begin{proof}
For any given integers $u, s>0$, 
we take an integer $k_{2}$ which satisfies
\[k_{2}^{2}>u+s.\]
Moreover,  we write
\[\mathbb{B}^{\rm ss}(s; \underline{w})=\left\{x\in \mathbb{B}\ :\  g^{-i}(x)\in \mathbb{V}_{w_{i}},\
i=1,\ldots,s\right\}.\]

Hereafter,  we suppose that the code condition \eqref{code3} holds for every 
\[k\geq \max\{k_{1}, k_{2}\},\] 
where $k_{1}$ is the integer given in Theorem \ref{thmWD}.
Consider 
the wandering domain $\mathbb{D}_{k}$ 
with \eqref{eq8-1}.
First, observe that, for any integers $u, s>0$,
\[P_{g}\in \mathbb{B}^{\rm u}(u; \underline{0}^{(u)})\cap\mathbb{B}^{\rm ss}(s; \underline{0}^{(s)}),\]
where $P_{g}$ is the saddle fixed point of the cs-blender horseshoe $\Lambda_{g}$ and 
$\mathbb{B}^{\rm u}(u; \underline{0}^{(u)})$ 
is the u-bridge given in Subsection \ref{ss4.1}.

Next, verify the following facts: 
\begin{itemize}
\item It follows 
from Proposition \ref{proj-w} that 
\[\mathbb{D}_{k}\subset 
\mathbb{G}^{\rm u}(\widehat n_{k},\underline{\widehat w}^{(k)}),
\]
where 
$\mathbb{G}^{\rm u}(\widehat n_{k},\underline{\widehat w}^{(k)})$ 
is the u-gap given in Subsection \ref{ss4.1}.
\item Thus, under the code condition \eqref{code3}, 
at least $k^{2}-(u+s)$ of the first $\widehat n_{k}-1$ iterates of $\mathbb{D}_{k}$ 
 are contained in the interior of 
$\mathbb{B}^{\rm u}(u; \underline{0}^{u})\cap\mathbb{B}^{\rm ss}(s; \underline{0}^{s})$.
\end{itemize}
The corresponding fact for a 2-dimensional  horseshoe  is stated by Colli-Vargas 
 \cite[Section 9]{CV01}.
The cs-blender horseshoe, on the other hand, 
is a 3-dimensional object, but this fact is actually the same as for the 2-dimensional horseshoe.
Based on this, the following evaluations will be conducted.

For any integer $\widehat k>0$, we here define  
\[
\widehat N_{\widehat k}=\sum_{i=0}^{\widehat k} (\widehat n_{k+i}+2).
\]
Then for any 
$x\in \mathbb{D}_{k}$, the empirical probability measure 
from $\widehat N_{\widehat k-1}$ to $\widehat N_{\widehat k}$
is 
\[
\nu_{x}^{(\widehat N_{\widehat k-1}, \widehat N_{\widehat k})}
=\frac{1}{\widehat N_{\widehat k}-\widehat N_{\widehat k-1}}
\sum_{i=\widehat N_{\widehat k-1}}^{\widehat N_{\widehat k}-1}\delta_{g^{i}(x)}
=\frac{1}{\widehat n_{k+\widehat k}+2}
\sum_{i=\widehat N_{\widehat k-1}}^{\widehat N_{\widehat k}-1}\delta_{g^{i}(x)}.
\]
Hence, 
for any $\varPhi \in C^{0}(\widehat{\mathbb{B}},\mathbb{R})$, 
\[
\int \varPhi d\nu_{x}^{(\widehat N_{\widehat k-1}, \widehat N_{\widehat k})}
=\frac{1}{\widehat n_{k+\hat k}+2}
\sum_{i=\widehat N_{\widehat k-1}}^{\widehat N_{\widehat k}-1}\varPhi\circ g^{i}(x)
=\frac{1}{\widehat n_{k+\hat k}+2}
\sum_{j=0}^{\widehat n_{k+\widehat k}+1}\varPhi\circ g^{j}(x_{\widehat N_{\widehat k-1}}),
\]
where $x_{\widehat N_{\widehat k-1}}=g^{\widehat N_{\widehat k-1}}(x)$.
Taking \eqref{eq8-3} into account, 
let us divide the sum from $0$ to $\widehat n_{k+\widehat k}+1$ into the following three parts:
\begin{multline*}
\sum_{j=0}^{\widehat n_{k+\widehat k}+1}\varPhi\circ g^{j}(x_{\widehat N_{\widehat k-1}})
=
\sum_{j=0}^{\widehat K+u-1} \varPhi\circ g^{j}(x_{\widehat N_{\widehat k-1}})\\
+
\sum_{j=\widehat K+u}^{\widehat K+(k+\widehat k)^{2}-s} \varPhi\circ g^{j}(x_{\widehat N_{\widehat k-1}})
+
\sum_{j=\widehat K+(k+\widehat k)^{2}-s+1}^{\widehat n_{k+\widehat k}+1}
\varPhi\circ g^{j}(x_{\widehat N_{\widehat k-1}}),
\end{multline*}
where $\widehat K=n_{0}+(k+\widehat k)+L(k+\widehat k)$. Note that 
\[\widehat n_{k+\widehat k}=\widehat K+(k+\widehat k)^{2}+m_{k+\widehat k}
=(k+\widehat k)^{2}+O(k+\widehat k).\]
Therefore, 
for any $\varepsilon>0$, 
there exist integers $u_{0}$, $s_{0}>0$ such that, 
for any $u\geq u_{0}$ and  $s\geq s_{0}$, 
\begin{multline*}
\sum_{j=0}^{\widehat n_{k+\widehat k}+1}\varPhi\circ g^{j}(x_{\widehat N_{\widehat k-1}})
\leq
(\widehat K+u)\|\varPhi\|_{C^{0}}\\
+
\left((k+\widehat k)^{2}-s-u+1\right)\left(\varPhi(P_{g})+\varepsilon \right)
+
(m_{k+\widehat k}+s)\|\varPhi\|_{C^{0}},
\end{multline*}
and 
\[
\sum_{j=0}^{\widehat n_{k+\widehat k}+1}\varPhi\circ g^{j}(x_{\widehat N_{\widehat k-1}})
\geq
\left((k+\widehat k)^{2}-s-u+1\right)
(\varPhi(P_{g})-\varepsilon).
\]

In consequence,  it follows from the code condition \eqref{code3} that  
for any sufficiently large $\widehat k$, 
\[
\varPhi(P_{g})-2\varepsilon
\leq
\frac{1}{\widehat n_{k+\widehat k}+2}
\sum_{j=0}^{\widehat n_{k+\widehat k}+1}\varPhi\circ g^{j}(x_{\widehat N_{\widehat k-1}})
\leq
\varPhi(P_{g})+2\varepsilon,
\]
and hence
\[
\left| \int \varPhi d\nu_{x}^{(\widehat N_{\widehat k-1}, \widehat N_{\widehat k})} -
\int \varPhi d\delta_{P_{g}}
\right|<2\varepsilon.
\]
That is, 
$\nu_{x}^{(\widehat N_{\widehat k-1}, \widehat N_{\widehat k})}$ 
converges to $\delta_{P_{g}}$ as $\widehat k\to +\infty$ 
in the weak*-topology.
\end{proof}

\begin{rmk}\label{tg}
Instead of \eqref{code3}, for any positive integer $k$ and 
$n\geq 2$, 
consider a $n$-periodic itinerary such that 
\[
\underline{v}^{(k)}
=\overbrace{
\underbrace{000\ldots01}_{n}
\underbrace{000\ldots01}_{n}
\ldots\ldots
\underbrace{000\ldots01}_{n}
\underbrace{000\ldots00}_{k^{2}-\lfloor k^{2}/n\rfloor n}
}^{k^{2}},
\]
where$\lfloor\cdot\rfloor$ stands for the floor function.
Since 
\[\lim_{k\to+\infty}\frac{k^{2}-\lfloor k^{2}/n\rfloor n}{k^{2}}=0,\]
$\underline{v}^{(k)}$ still satisfies both \eqref{k2} and \eqref{dom}.
Then, for such a $\underline{v}^{(k)}$, by the same procedure as in proof of Theorem \ref{phys}, 
one can obtain the non-trivial Dirac physical measure associated with the wandering domain 
supported 
by the $n$-periodic orbit.
\end{rmk}

\appendix
\section{}\label{apdx}


To show Lemma \ref{rt}, we only need to verify the existence of a
folding manifold which is contained in $\smfd(\Lambda)$, 
because it implies that $f$ has 
a $C^{1}$-robust homoclinic tangency of $\Lambda$ from \cite[Theorem 4.8]{BD12}.
Here the \emph{folding manifold} of $\Lambda$ is a 2-dimensional manifold with 
the following conditions:
\begin{itemize}
\item $\mathcal{S}=\bigcup_{t\in[t_{-},t_{+}]}\mathcal{S}_t$,  where $t_{\pm}\in \mathbb{R}$ with $t_{-}<t_{+}$
and  
$\mathcal{S}_{t}$ is 
a (1-dimensional) ss-disc of $\mathbb{B}$;
\item both $\mathcal{S}_{t_{-}}$ and $\mathcal{S}_{t_{+}}$ intersect $\lumfd(P)$;
\item for any $t\in (t_{-},t_{+})$, $\mathcal{S}_{t}$ lies  between $\lumfd(P)$ and $\lumfd(Q)$.
\end{itemize}
\begin{proof}[Proof of Lemma \ref{rt}]
Let $\ell_{0}$ and $\ell_{1}$ be the 
parallel edges of $\mathbb{B}$ given as 
\[\ell_{0}=[1/2-\delta, 1/2+\delta]\times(0,0),\  
\ell_{1}=[1/2-\delta, 1/2+\delta]\times(1,1).\] 
 By \eqref{tang},  
$\tilde \ell_{0}=f^{2}(\ell_{0})$ 
and 
$\tilde \ell_{1}=f^{2}(\ell_{1})$ 
are  contained in the quadratic curves, respectively, as  
\[
\left\{
 x=-a_{1}a_{4}^{-2}\Bigl(z-\frac{1}{2}\Bigr)^{2}, y=\frac{1}{2}-\frac{a_{3}}{2}
\right\},\ 
\left\{
x=-a_{1}a_{4}^{-2}\Bigl(z-\frac{1}{2}\Bigr)^{2}+a_{2}, y=\frac{1}{2}+\frac{a_{3}}{2}
\right\}.
\]
See Figure 
\ref{tang_fig}. 
 \begin{figure}[hbt]
\centering
\scalebox{0.6}{
\includegraphics[clip]{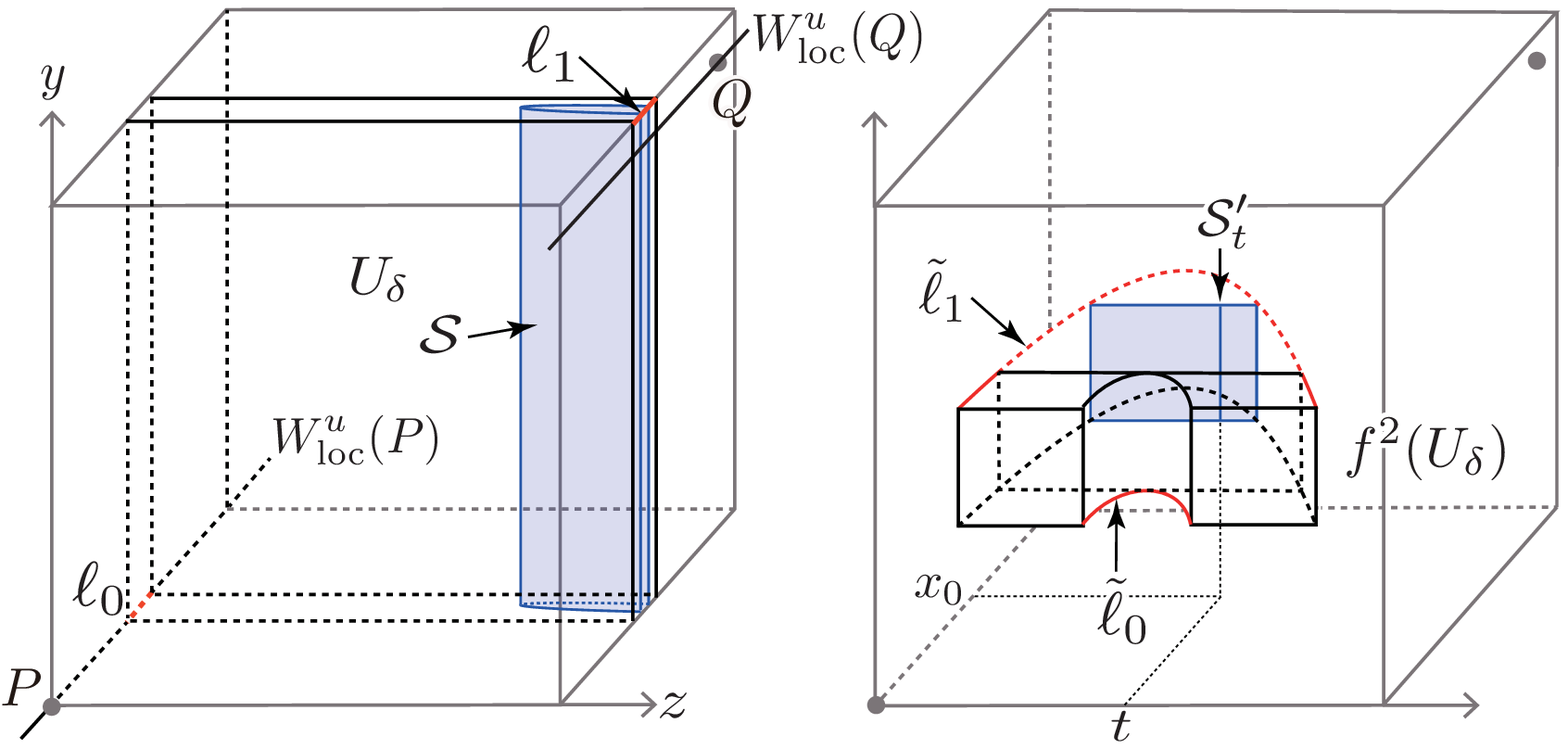}
}
\caption{} 
\label{tang_fig}
\end{figure}

For a given $0<x_{0}< a_{2}$, 
we write 
\[t_{\pm}(x_{0})=\frac{1}{2}\pm \sqrt{a_{1}^{-1}a_{4}(a_{2}-x_{0})}.\] 
By the second condition in \eqref{a_i}, the value inside the root is positive.
For any $t\in [t_{-}(x_{0}), t_{+}(x_{0})]$, 
consider the vertical ss-disc defined  as 
\[
\mathcal{S}_{t}^{\prime}=\mathcal{S}_{t}^{\prime}(x_{0})=\{ x_{0}\}\times \Bigl[\frac{1}{2}-\frac{a_{3}}{2}, \frac{1}{2}+\frac{a_{3}}{2}\Bigr]\times \{t\}.
\]
Note that 
$\mathcal{S}_{t}^{\prime}(x_{0})$ can be contained in
$\lsmfd(\Lambda)$ if one chooses $x_0$ appropriately. 
Let 
$\mathcal{S}^{\prime}$ be 
the collection of all of $\mathcal{S}_{t}^{\prime}$ with $t\in [t_{-}(x_{0}), t_{+}(x_{0})]$. 
Observe that the intersection of 
$\mathcal{S}^{\prime}$ and $\tilde \ell_{0}$ consists of 
two transverse points.  
It 
implies that 
 $\mathcal{S}_{t_{-}(x_{0})}$ and  $\mathcal{S}_{t_{+}(x_{0})}$ intersect $\lumfd(Q)$. 
 Moreover, it follows from \eqref{blender} and \eqref{tang} 
 that $f$ preserves the $y$-direction. Thus
 $\mathcal{S}_{t}=f^{-1}(\mathcal{S}_{t}^{\prime})$ 
is an ss-disc. 
In consequence, 
 $\mathcal{S}=f^{-1}(\mathcal{S}^{\prime})$ 
 is a folding stable manifold of $\Lambda$.
\end{proof}

\section*{Acknowledgements}
The authors thank Pablo~G.~Barrientos and  Artem Raibekas 
for their comments and suggestions during the planning stages of this study.

\bibliographystyle{amsalpha}

\newcommand{\etalchar}[1]{$^{#1}$}
\def\cprime{$'$}
\providecommand{\bysame}{\leavevmode\hbox to3em{\hrulefill}\thinspace}
\providecommand{\MR}{\relax\ifhmode\unskip\space\fi MR }
\providecommand{\MRhref}[2]{%
  \href{http://www.ams.org/mathscinet-getitem?mr=#1}{#2}
}
\providecommand{\href}[2]{#2}

\end{document}